\documentclass{amsart}

\usepackage{amsmath, amssymb, amsfonts, amsthm, url, graphicx}

\usepackage{amssymb,amscd,amsthm,mathrsfs}
\usepackage[ansinew]{inputenc}
\usepackage[centertags]{amsmath}
\usepackage[all]{xypic}
\usepackage{graphicx}
\usepackage{color}

   \def\DD{{\mathbb D}}

 \def\RR{{\mathbb R}}

\newtheorem*{teo*}{Theorem}

\newtheorem{teo}{Theorem}[section]

\newtheorem{lema}[teo]{Lemma}
\newtheorem{prop}[teo]{Proposition}

\newcommand{\bi}{\begin{itemize}}
	\newcommand{\ei}{\end{itemize}}

\theoremstyle{definition}

\theoremstyle{remark}

\newcommand{\eps}{\varepsilon}

\newcommand{\comment}[1]{}

\begin{document}
	
	\newtheorem{theorem}{Theorem}[section]
	\newtheorem{lemma}[theorem]{Lemma}
	\newtheorem{proposition}[theorem]{Proposition}
	\newtheorem{corollary}[theorem]{Corollary}
	
	\theoremstyle{definition}
	\newtheorem{definition}[theorem]{Definition}
	\newtheorem{remark}[theorem]{Remark}
	\newtheorem{example}[theorem]{Example}

	\title[Stability conjecture for geodesic flows without conjugate points]{The stability conjecture for geodesic flows of compact manifolds without conjugate points and quasi-convex universal covering}
\author{Rafael Potrie}
	\address{Centro de Matem\'{a}ticas, UDELAR
		Montevideo, Uruguay} \email{rpotrie@cmat.edu.uy}
	\thanks{ R.P. is partially supported by C.S.I.C.  }
	\author{Rafael O. Ruggiero}
	\address{Departamento de Matem\'atica, Pontif\'{\i}cia Universidade Cat\'olica do Rio de Janeiro,
		Rio de Janeiro, RJ, Brazil, 22453-900} \email{rorr@mat.puc-rio.br}
	\thanks{The research project of R.O.R. is partially supported by FAPERJ (Cientistas do nosso estado, Programa de Apoio a Projetos Tem\'{a}ticos), Pronex de Geometria n. E26/010.001249/2016, (Brazil)   }
	
	
	\date{November 20th, 2023}
	\keywords{Geodesic flow, conjugate points,quasi-convex universal covering,  hyperbolic set, Anosov flow}

	\begin{abstract}
		Let $(M,g)$ be a $C^{\infty}$ compact, boudaryless connected manifold without conjugate points with quasi-convex universal covering and divergent geodesic rays. 
		We show that the geodesic flow of $(M,g)$ is $C^{2}$-structurally stable from Ma\~{n}\'{e}'s viewpoint if and only if  it is an Anosov flow, proving the so-called $C^{1}$-stability conjecture.  
		
	\end{abstract}
	\maketitle

	\section{Introduction}
	
	The $C^{1}$-stability conjecture for geodesic flows remains a challenging open problem in dynamics until present  because Pugh's $C^{1}$-closing lemma \cite{kn:Pugh} is not known for the category of geodesic flows. The $C^{1}$-structural stability conjecture for geodesic flows states that a $C^{1}$-structurally stable geodesic flow is an Anosov flow. Every Anosov geodesic flow acting on a compact Riemannian manifold is of course $C^{1}$-structurally stable by the celebrated work of Anosov \cite{kn:Anosov}, the $C^{1}$ stability conjecture is the converse of Anosov's result. Here, one needs to make some precisions on regularity: by $C^1$-structurally stable geodesic flow of a metric $g$ we mean that there is a $C^2$-open set of metrics containing $g$ so that for every $g'$ in the open set we have that the geodesic flows of $g$ and $g'$ are orbit equivalent; note that a $C^2$-neighborhood of $g$ provides perturbations of the geodesic vector field which are $C^1$ (but not all possible such perturbations). Some places in the literature, for instance \cite{kn:CM} use $C^2$-structural stability to refer to this same notion, but we keep the definition as in \cite{kn:ARR,kn:Ruggiero-AHL}.

The proof of the $C^{1}$ stability conjecture for diffeomorphisms of compact manifolds is due to Ma\~{n}\'{e} \cite{kn:Mane,kn:Mane88} for arbitrary dimensions and independently by Liao \cite{kn:Liao} for surfaces. Ma\~{n}\'{e}'s proof has roughly two main parts: first of all the $C^{1}$ structural stability implies that the closure of the set of hyperbolic periodic orbits is a hyperbolic set for the diffeomorphism; the second part is to show that the nonwandering set of the diffeomorphism is the closure of the set of periodic orbits. The first part was extended to geodesic flows in the more general context of Ma\~{n}\'{e} type perturbations, giving rise to the concept of $C^{2}$ structural stability from Ma\~{n}\'{e}'s viewpoint (see for instance \cite{kn:ARR}). Recent developments of the theory of contact homology for contact flows in 3-manifolds \cite{kn:Irie} yield that the set of periodic orbits of the geodesic flow of a $C^{\infty}$ generic metric on a compact surface are dense after projected to the surface. This is not enough to ensure density in the unit tangent bundle which would allow, combined with \cite{kn:ARR}) would get that $C^{1}$ structurally stable geodesic flows (in the sense that a $C^2$-open neighborhood of metrics have conjugated geodesic flow) of compact surfaces must be Anosov flows. Using another approach that allowed them to circunvent this, in \cite{kn:CM} the $C^1$-structural stability conjecture for geodesic flows was obtained. The extension of their argument to higher dimensions is far from reach at present. 
	
	Recent results indicate that the conjecture for compact manifolds without conjugate points may follow from the theory of the global geometry of such manifolds \cite{kn:Ruggiero-TAMS}, \cite{kn:Ruggiero-AHL} that would play in some sense the role of the missing $C^{1}$-closing lemma. The main result of this article is a proof of the $C^{1}$-stability conjecture for an important category of geodesic flows of compact manifolds without conjugate points. 
	\bigskip
	
	\begin{theorem} \label{main}
		Let $(M^{n},g)$ be a compact $C^{\infty}$, $n$-dimensional, connected Riemannian manifold without conjugate points  such that the universal covering is a quasi-convex space where
		geodesic rays diverge. Then, the geodesic flow is $C^{2}$ structurally stable from Ma\~{n}\'{e}'s viewpoint if and only if the geodesic flow is Anosov.
	\end{theorem}
	\bigskip
	
	Theorem \ref{main} extends the main result of \cite{kn:Ruggiero-AHL}, where the same statement is proved assuming that the manifold has dimension 2 or 3. Let us describe briefly the ideas of the proof in low dimensions in order to give a context to Theorem \ref{main}. The key steps of the proof in the 3-dimensional case are the following. The first step is to show that the $C^{2}$-structural stability from Ma\~{n}\'{e}'s viewpoint implies that the closure of the set of periodic orbits is a hyperbolic set for the geodesic flow for every $n \in \mathbb{N}$ \cite{kn:ARR}. The second step is the proof of a topological generalization of the so-called flat strip Theorem assuming that the universal covering $(\tilde{M}, \tilde{g})$ is a quasi-convex space where geodesic rays diverge \cite{kn:Ruggiero-AHL}. The generalization of the flat strip Theorem combined with the hyperbolicity of the closure of periodic orbits implies that the fundamental group of $M^{n}$ is a Preissmann group: every abelian subgroup is infinite cyclic. The Preissmann property and some metric properties of abelian subgroups of the fundamental group of manifolds without conjugate points \cite{kn:CS} lead to the the third step of the proof : the application of Thurston's geometrization theory for 3-dimensional manifolds to conclude that the manifold admits a Riemannian structure of negative curvature. The combination of this fact and the generalized flat strip Theorem yield that closed orbits are dense in $T_{1}M$, which implies that the geodesic flow is Anosov. In the case of surfaces the density of closed orbits follows without the need of Thurston's geometrization. 
	
	The strategy of the proof of Theorem \ref{main} follows the first and the second steps of the proof of the 3-dimensional case, which hold in any dimension. However, the third step does not extend to higher dimensional manifolds. The main contribution of this article is to explore the existence of continuous invariant foliations (see Corollary \ref{fol-continuity}) for the geodesic flow under the assumptions of Theorem \ref{main}, combined with the hyperbolicity of the closure of the set of hyperbolic orbits, to show the density of periodic orbits. Notice that the argument used to prove Theorem \ref{main} in the 3-dimensional case gives more information about the manifold $M$ than the argument used in higher dimensions: $M^{3}$ admits a Riemannian structure of negative curvature. As far as we know, the existence of a Riemannian metric of negative curvature in a compact Riemannian manifold whose geodesic flow is Anosov is an open problem. 
	
	The structure of the paper is the following: from Sections 2 to 4 we introduce some preliminaries of the theory of Riemannian manifolds without conjugate points; we make a survey of some basic definitions and results about hyperbolic dynamics; and present what is known about the structure of the "non-expansive" set of the geodesic flow: the so-called generalized strips. This is the main issue of the article: if we knew that there were no generalized strips of the geodesic flow the proof of Theorem \ref{main} would be relatively straightforward from already know results. In Section 5 we show the existence of generalized local product neighborhoods for every point in an open neighborhood of the closure of the set of hyperbolic periodic points. This result is interesting in itself since it shows somehow that the transversality of the invariant foliations of the geodesic flow, the horospherical foliations, is not necessary to have a local product structure in an open neighborhood of a point. From Sections 6 on we define some foliated cross sections of the geodesic flow and some foliated subsets to study the action of the Poincar\'{e} map of the section on these sets. Roughly speaking, we show that the Poincar\'{e} map has a sort of topologically hyperbolic behavior restricted to such sets, in analogy to what happens with foliated neighborhoods in hyperbolic dynamics. Then, applying Lefchetz fixed point Theorem we deduce that the generalized strip of a recurrent point  in a small neighborhood of the closure of the set of closed orbits must be arbitrarilly approached by periodic orbits, a sort of closing lemma for recurrent strips. From this and the hyperbolicity of the closure of periodic orbits we finally conclude the proof of Theorem \ref{main} in Section 7. 

As a further application of the ideas developed in the article, we show in Section 8 the following statement that is interesting in its own:

\begin{theorem} \label{main2}
Let $(M^{n},g)$ be a compact $C^{\infty}$, $n$-dimensional, connected Riemannian manifold without conjugate points  such that the universal covering is a quasi-convex space where geodesic rays diverge. Suppose that the geodesic flow has a hyperbolic periodic orbit. Then there exists a topological horseshoe accumulating the periodic orbit, and in particular, the topological entropy of the geodesic flow is positive. 
\end{theorem}

The relationship between the topological entropy of the geodesic flow and the
growth of the fundamental group was pointed out by Ma\~{n}\'{e}-Freire [12]: they show
that the topological entropy is equal to the volume growth of balls in the universal
covering, and the volume growth of balls coincides with the growth of the fundamental group (Milnor \cite{kn:Milnor}, Manning \cite{kn:Manning}, Dinaburg  \cite{kn:Dinaburg}). We would like to remark that
the assumption of the existence of a hyperbolic periodic orbit already implies that
the volume growth of balls in the universal covering is super polynomial if $(M,g)$
is a compact, analytic manifold without conjugate points. Because by the work
of Gromov \cite{kn:Gromov}, when the growth of balls is polynomial the fundamental group is
virtually nilpotent; and by Croke-Schroeder \cite{kn:CS}, every virtually nilpotent subgroup
of the fundamental group of a compact, analytic manifold without conjugate points
must be abelian. Therefore, if a compact, analytic $(M,g)$ has no conjugate points
and the growth of balls in the universal covering is polynomial, the manifold is
actually a torus. By Hopf \cite{kn:Hopf}, Burago-Ivanov \cite{kn:BI}, the manifold is flat, and hence
without hyperbolic periodic orbits. Theorem 1.2 improves this remark since super polynomial growth of a discrete, finitely generated group might
not imply exponential growth, that must be the case of the fundamental group of
the manifold when the topological entropy of the geodesic flow is positive. 

\medskip

{\small \emph{Acknowledgements:} The authors would like to thank Gonzalo Contreras for useful discussions and comments. }

	\section{Preliminaries}\label{SECpreliminaries}
	
	Let us start by introducting some basic  notations and objects of the theory of manifolds without conjugate points. We follow mainly \cite{kn:Eberlein,kn:Ruggiero-Ensaios,kn:Ruggiero-AHL}. 
	
	A pair $(M,g)$ denotes a $C^{\infty}$ complete, connected Riemannan manifold, $TM$ will denote its
	tangent space, $T_{1}M$ denotes its unit tangent bundle. $\Pi: TM \longrightarrow M$ denotes the canonical projection $\Pi(p,v) = p$, the coordinates
	$(p,v)$ for $TM$ will be called canonical coordinates. The universal covering of $M$ is $\tilde{M}$, the covering map is denoted by $\pi:\tilde{M} \longrightarrow M$,
	the pullback of the metric $g$ by $\pi$ is denoted by $\tilde{g}$. The geodesic $\gamma_{(p,v)}$ of $(M,g)$ or $(\tilde{M},\tilde{g})$ is the unique geodesic whose initial conditions
	are $\gamma_{(p,v)}(0)=p$, $\gamma_{(p,v)}'(0)= v$. All geodesics will be parametrized by arc length unless explicitly stated. 
	
	The covering map $\pi $ induces a natural covering map $P : T_{1}\tilde{M} \longrightarrow T_{1}M$ given by $ (\tilde{p}, \tilde{v}) \rightarrow (\pi(\tilde{p}), d\pi_{\tilde{p}}(\tilde{v}))$. 
	
	\begin{definition}
		A $C^{\infty}$ Riemannian manifold $(M,g)$ has no conjugate points if the exponential map is nonsingular at every point.
	\end{definition}
	
	Nonpositive curvature manifolds are the best known examples of manifolds without conjugate points, but there are of course many examples of manifolds without conjugate points whose  sectional curvatures change sign (see for instance \cite{kn:Burns}). 
	
	\begin{definition} \label{divergence}
		We say that geodesic rays diverge uniformly, or simply diverge in $(\tilde{M}, \tilde{g})$ if for given $\epsilon >0$, $T>0$, there exists $R>0$ such that for any given $p \in \tilde{M}$
		and two different geodesic rays $\gamma:\mathbb{R} \longrightarrow \tilde{M}$,
		$\beta : \mathbb{R} \longrightarrow \tilde{M}$ with $\gamma(0)= \beta(0) =p$, subtending an angle at $p$ greater than $\epsilon$, then $d(\gamma(t),\beta(t)) \geq T$
		for every $t \geq R$.
	\end{definition}
	
	This condition is quite common in all known examples of manifolds without conjugate points 
	(compact surfaces without conjugate points \cite{kn:Green}, nonpositive curvature manifolds, manifolds without focal points and more generally, manifolds with continuous Green bundles \cite{kn:Ruggiero-TAMS}). However, it is not known whether it is satisfied for every compact manifold without conjugate points and no further assumptions.
	
	\begin{definition} \label{busemann}
		Given $\theta = (p,v) \in T_{1}\tilde{M}$, the Busemann function $b^{\theta} : \tilde{M} \longrightarrow \mathbb{R}$ is given by
		$$ b^{\theta}(x) = \lim_{t\rightarrow +\infty} (d(x,\gamma_{\theta}(t)) -t) .$$
	\end{definition}
	
	Busemann functions of compact manifolds without conjugate points are $C^{1,k}$, namely, $C^{1}$ functions with $k$-Lipschitz first derivatives
	where the constant $k$ depends on the minimum value of the sectional curvatures (see for instance \cite{kn:Pesin}, Section 6).
	The level sets of $b^{\theta}$, $H_{\theta}(t)$, are called \textit{horospheres}, the gradient $\nabla b^{\theta}$ is a Lipschitz unit vector field and
	its flow preserves the foliation of the horospheres $H_{\theta}$ that is a continuous foliation by $C^{1}$, equidistant leaves. The integral orbits of
	the Busemann flow are geodesics of $(\tilde{M},\tilde{g})$ and usually called \textit{Busemann asymptotes} of $\gamma_{\theta}$.  When the curvature is nonpositive, Busemann functions and horospheres are $C^{2}$ smooth \cite{kn:Eschenburg}.
	
	\subsection{Horospheres and center stable/unstable sets for the geodesic flow}
	
	\begin{definition} \label{centerstable}
		For $\theta = (p,v) \in T_{1}\tilde{M}$ let
		$$\tilde{\mathcal{F}}^{s}(\theta) = \{ (x, -\nabla_{x}b^{\theta}), \mbox{ } x \in H_{\theta}(0) \}$$
		$$\tilde{\mathcal{F}}^{u}(\theta) = \{ (x, \nabla_{x}b^{(p,-v)}), \mbox{ } x \in H_{(p,-v)}(0) \}.$$
		$$ \tilde{\mathcal{F}}^{cs}(\theta) = \cup_{t \in \mathbb{R}} \phi_{t}(\tilde{\mathcal{F}}^{s}(\theta))$$ 
		$$ \tilde{\mathcal{F}}^{cu}(\theta) = \cup_{t \in \mathbb{R}} \phi_{t}(\tilde{\mathcal{F}}^{u}(\theta))$$ 
		Similar objects can be defined in $T_{1}M$, let
		$$\mathcal{F}^{s}(P(\theta)) = P(\tilde{\mathcal{F}}^{s}(\theta))$$
		$$\mathcal{F}^{u}(P(\theta)) = P(\tilde{\mathcal{F}}^{u}(\theta)).$$
		$$ \mathcal{F}^{cs}(P(\theta)) = P(\tilde{\mathcal{F}}^{cs}(\theta))$$
		$$ \mathcal{F}^{cu}(P(\theta)) = P(\tilde{\mathcal{F}}^{cu}(\theta))$$
		Let us denote by $\mathcal{F}^{s}$, $\tilde{\mathcal{F}}^{s}$ the collection of the sets $\mathcal{F}^{s}(\theta)$, $\theta \in T_{1}M$, $\tilde{\mathcal{F}}^{s}(\theta)$, $\theta \in T_{1}\tilde{M}$, and by $\mathcal{F}^{u}$, $\tilde{\mathcal{F}}^{u}$ the collection of the sets $\mathcal{F}^{u}(\theta)$, $\tilde{\mathcal{F}}^{u}(\theta)$. Let us denote similarly the collections of sets $\mathcal{F}^{cs}$ $\mathcal{F}^{cu}$, $\tilde{\mathcal{F}}^{cs}$, $\tilde{\mathcal{F}}^{cu}$. 
	\end{definition}
	
	The sets $\mathcal{F}^{s}(\theta)$, $\mathcal{F}^{u}(\theta)$ coincide with the stable and unstable submanifold respectively, of the geodesic flow when it is Anosov (see Section 3 for more details). The sets $\mathcal{F}^{cs}(\theta)$, $\mathcal{F}^{cu}(\theta)$ coincide with the center stable and center unstable submanifolds respectively. In general, it is not know wether such collections of sets are foliations. Some regularity properties of the above sets are described in the following result. 
	
	Consider the vertical fiber $V(\theta) = \{ (p,w), \mbox{ }, w \in T_{p}\tilde{M}, \mbox{ },\parallel w \parallel =1\}\}$

	\begin{theorem} \label{Lipschitz} (Eberlein \cite{kn:Eberlein-73}, Pesin \cite{kn:Pesin})
		
		Let $(M,g)$ be a complete Riemannian manifold without conjugate points such that the sectional curvatures of $(M,g)$ are bounded from below by a constant $c\leq 0$. Then there exists $L>0$, depending on $c$, $\epsilon >0$, and an open neighborhood $B_{\epsilon}(\theta)$ such that  such that each subset $\tilde{\mathcal{F}}^{s}(\theta)$, $\tilde{\mathcal{F}}^{u}(\theta)$ is a $(n-1)$-$L$-Lipschitz continuous submanifold of $T_{1}\tilde{M}$. Moreover, each of these submanifolds is topologically transverse to the vertical fiber $V(\theta) = \{ (p,w), \mbox{ }, w \in T_{p}\tilde{M}, \mbox{ },\parallel w \parallel =1\}\}$, in the following sense: we have that  
		$$\tilde{\mathcal{F}}^{s}(\theta) \cap V(\theta) =\{ \theta \}= \tilde{\mathcal{F}}^{u}(\theta) \cap V(\theta)$$ 
		for every $\theta \in T_{1}\tilde{M}$. 
	\end{theorem} 
	
	We shall come back later (Section 4) to explore this sort of "topological transversality" between invariant submanifolds and vertical fibers to construct some special open foliated neighborhoods for the dynamics.

	\subsection{Divergence of geodesic rays and uniform continuity of large spheres and horospheres}
	
	We review in the subsection some relevant results linking the uniform divergence of geodesic rays in the universal covering of a compact manifold without conjugate points and the uniform behavior of the global geometry of large spheres and horospheres. We follow \cite{kn:Ruggiero-Ast}. Let 
$$ S_{T}(p) = \{ x \in \tilde{M}, \mbox{ }, d(x,p) = T \} $$ 
be the sphere of radius $T$ around $p \in \tilde{M}$, and let 
$$ B_{T}(p) = \{ x \in \tilde{M}, \mbox{ }, d(x,p) \leq T \} $$  
be the closed ball of radius $T$ around $p \in \tilde{M}$. 
	
	\begin{theorem} \label{h-continuity}
		Let $(M,g)$ be a compact Riemannian manifold without conjugate  points such that geodesic rays diverge in $(\tilde{M}, \tilde{g})$. Then given $\epsilon >0$ and $R>0$, there exists $T>0$ such that for every $\theta = (p,v) \in T_{1}\tilde{M}$, the $C^{1}$ distance from the sets $S_{T}(\gamma_{\theta}(T)) \cap B_{R}(p)$ and $H_{\theta}(0) \cap B_{R}(p)$ is less than $\epsilon$. Moreover, there exists $\delta >0$ such that if $d(\theta, \eta) < \delta$ then the $C^{1}$ distance 
		from $H_{\theta}(0)\cap B_{R}(p)$ and $H_{\eta}(0)\cap B_{R}(p)$ is less than $\epsilon$. 
	\end{theorem}
	
	Let $b^{\theta}_{T}: \tilde{M} \longrightarrow \mathbb{R}$ be defined by 
	$b^{\theta}_{T} (x) = d(x, \gamma_{\theta}(T)) - T$. By definition, the Busemann function $b^{\theta}(x)$ is the limit as $T \rightarrow +\infty$ of $b^{\theta}_{T}(x)$. The level sets of $b^{\theta}_{T}(x)$ are the spheres $S_{T}(\gamma_{\theta}(T))$. So Theorem \ref{h-continuity} implies that the functions $b^{\theta}_{T}$ converge to the functions $b^{\theta}$ uniformly on compact subsets. Since spheres in $\tilde{M}$ are $C^{\infty}$ submanifolds, the set 
	$$\tilde{\mathcal{F}}^{s}_{T}(\theta) = \{ (x, -\nabla_{x}b^{\theta}_{T}), \mbox{ }x \in S_{T}(\gamma_{\theta}(T))\} $$ 
	is a $C^{\infty}$ submanifold of $T_{1}\tilde{M}$. As a consequence of Theorem \ref{h-continuity} and Theorem \ref{Lipschitz} we have 
	
	\begin{corollary} \label{fol-continuity}
		Under the assumptions of Theorem \ref{h-continuity}, the submanifolds $\tilde{\mathcal{F}}^{s}_{T}(\theta)$ converge uniformly in the $C^{0}$ topology on compact sets to $\tilde{\mathcal{F}}^{s}(\theta)$. Moreover, the sets $\tilde{\mathcal{F}}^{s}(\theta)$ form a continuous foliation $\tilde{\mathcal{F}}^{s}$ of $T_{1}\tilde{M}$ by $n$-dimensional leaves, as well as the sets $\tilde{\mathcal{F}}^{u}(\theta)$ which form a continuous foliation $\tilde{\mathcal{F}}^{u}$. In particular, the collection of sets $\mathcal{F}^{s}(\theta)$, $\mathcal{F}^{u}(\theta)$  with $\theta \in T_{1}M$ for continuous foliations $\mathcal{F}^{s}$, $\mathcal{F}^{u}$ by Lipschitz leaves. 
	\end{corollary}
	
	\subsection{Quasi-convexity} 
	
	\begin{definition} \label{quasi-convexity}
		The universal covering $(\tilde{M}, \tilde{g})$ of a complete Riemannian manifold $(M,g)$ is called a $(K,C)$-quasi-convex space, or simply a quasi-convex space, if there exist constants
		$K>0$, $C>0$, such that given two pairs of points $x_{1}, x_{2}$, $y_{1},y_{2}$ in $\tilde{M}$ and two minimizing geodesics $\gamma:[0,1] \longrightarrow \tilde{M}$,
		$\beta:[0,1] \longrightarrow \tilde{M}$ such that $\gamma(0)= x_{1}$, $\gamma(1)= y_{1}$, $\beta(0) = x_{2}$, $\beta(1)= y_{2}$, we have
		$$ d_{H}(\gamma, \beta) \leq K\max\{ d(x_{1},x_{2}), d(y_{1},y_{2})\} + C$$
		where $d_{H}$ is the Hausdorff distance.
	\end{definition}
	
	The universal covering of manifold of nonpositive sectional curvature is $(1,0)$-quasi-convex. Most of the known categories of manifolds without conjugate points (no focal points, bounded asymptote)
	have quasi-convex universal coverings. Moreover, by the work of Morse (\cite{kn:Morse}), when $(M,g)$ is a compact surface without conjugate points and genus greater than one, the universal covering $(\tilde{M}, \tilde{g})$ is a quasi-convex space. In higher dimensions, Gromov \cite{kn:Gromov} shows that the universal covering of every compact manifold whose fundamental group is hyperbolic is quasi-convex.
	
	\subsection{Busemann asymptoticity versus asymptoticity and the generalized flat strip Theorem}
	
	\begin{definition}
		A geodesic $\beta : \mathbb{R} \longrightarrow \tilde{M} $ is forward asymptotic to a geodesic $\gamma :\mathbb{R} \longrightarrow \tilde{M} $ if there exists $L>0$ such that
		$$d_{H}(\gamma[0,+\infty), \beta[0,+\infty)) \leq L$$
		where $\gamma[0,+\infty) = \{ \gamma(t), t \geq 0\}$. A geodesic $\sigma \subset \tilde{M}$ is backward asymptotic to $\gamma$ if there exists $L>0$ such that
		$$d_{H}(\gamma(-\infty,0], \beta(-\infty,0]) \leq L.$$
		Two geodesics $\gamma$, $\beta$ in $\tilde{M}$ are bi-asymptotic if they are both forward and backward asymptotic.
	\end{definition}
	
	In the case of manifolds of nonpositive curvature or no focal points, every Busemann asymptote of a geodesic $\gamma_{\theta} \subset \tilde{M}$ is forward asymptotic to $\gamma_{\theta}$. Without these assumptions, a Busemann asymptote of $\gamma_{\theta}$ might not be asymptotic to $\gamma_{\theta}$. The most general result about the subject is perhaps the following statement (see \cite{kn:Ruggiero-Ast}, or \cite[Lemma 2.9]{kn:Ruggiero-AHL}  for instance) 
	
	\begin{lemma} \label{asymptotic-1}
		Let $(M,g)$ be compact without conjugate points such that geodesic rays diverge in $\tilde{M}$. Then
		\begin{enumerate}
			\item Any geodesic $\gamma_{\eta}$ forward asymptotic to $\gamma_{\theta}$ is a Busemann asymptote of $\gamma_{\theta}$.
			Moreover, there exists $c>0$ such that $b^{\eta} (x) = b^{\theta}(x) +c$ for every $ x \in \tilde{M}$.
			\item In particular, any geodesic that is bi-asymptotic to $\gamma_{\theta}$ is a Busemann asymptote of
			$\gamma_{\theta}$ and $\gamma_{-\theta}$ where $\theta = (p,v)$, $-\theta= (p,-v)$. In this case, if $\gamma_{\eta}(0) \in H_{\theta}(0)$ then
			$\gamma_{\eta}(0) \in H_{\theta}(0)\cap H_{-\theta}(0)$
			\item Moreover, if $(\tilde{M}, \tilde{g})$ is $K,C$-quasi-convex, for every $\theta \in T_{1}\tilde{M}$ and $\eta \in \tilde{\mathcal{F}}^{s}(\theta)$, the geodesics $\gamma_{\theta}$ and $\gamma_{\eta}$ are forward asymptotic and satisfy
			$$ d(\gamma_{\theta}(t), \gamma_{\eta}(t)) \leq Kd(\gamma_{\theta}(0), \gamma_{\eta}(0)) + C $$
			for every $t \geq 0$. The same estimate holds for $\gamma_{\theta}$ and $\gamma_{\eta}$ if $\eta \in \tilde{\mathcal{F}}^{u}(\theta)$ for every $ t \leq 0$. 
		\end{enumerate}
	\end{lemma}
	
	The relationshing between the intersections $H_{\theta}(0)\cap H_{-\theta}(0)$ and bi-asymptotic geodesics  is one of the most intriguing problems in
	the theory of manifolds without conjugate points. In the case of compact surfaces such a set is a connected compact curve with boundary (that might be a single point of course).
	The convexity properties of spaces of nonpositive curvature yield that such a set is a convex flat set (see for instance \cite{kn:BGS})
	which generates a flat invariant set under the action of the geodesic flow. This result is a consequence of the well known flat strip Theorem, in the case of surfaces of higher genus 
	they are true strips, namely, isometric embeddings of $\mathbb{R}\times [a,b]$ endowed with the Euclidean metric. 
	
	Geodesics of simply connected (non-compact) manifolds with negative curvature do not have nontrivial strips, each strip is a single geodesic. This feature is equivalent to the notion of expansivity (see \cite{kn:Ruggiero-expansive} for definition and details). Expansivity has strong consequences, a compact manifold without conjugate points and expansive geodesic flow has Gromov fundamental group and its universal covering is a visibility manifold \cite{kn:Ruggiero-SBM}. 
	
	Without restrictions on the sign of sectional curvatures of a compact surface without conjugate points, the intersections between invariant submanifolds
	$\tilde{W}^{s}(\theta) \cap \tilde{W}^{u}(\theta)$ might be non-flat strips as shown by Burns (see \cite{kn:Burns}). In higher dimensions, it is very difficult to expect 
	any control over the geometry of  $H_{\theta}(0)\cap H_{-\theta}(0)$. 
	Such intersections might be even disconnected sets. However,  we have the following recent extension of the flat strip Theorem \cite{kn:Ruggiero-AHL}: 
	
	\begin{lemma} \label{strip}
		Let $(M,g)$ be a compact $C^{\infty}$ manifold without conjugate points such that $(\tilde{M}, \tilde{g})$ is $(K,C)$-quasi-convex where geodesic rays diverge.
		Then given $\theta = (p,v) \in T_{1}\tilde{M}$, and a geodesic $\beta=\gamma_{\eta}$ (with $\eta  \in T_{1}\tilde{M}$) bi-asymptotic to $\gamma = \gamma_{\theta}$, there exists a connected set $\Sigma(\gamma, \beta) \subset H_{\theta}(0) \cap H_{-\theta}(0)$ containing
		$p$ and $\beta \cap H_{\theta}(0)$, such that 
		\begin{enumerate}
			\item for every $x \in \Sigma(\gamma, \beta)$, the geodesic with initial conditions $(x, -\nabla_{x} b^{\theta})$ is bi-asymptotic to both of them.
			In particular, the set
			
			$$ S(\gamma , \beta) = \bigcup_{x \in \Sigma(\gamma, \beta), t \in \mathbb{R}} \gamma_{(x,-\nabla_{x} b^{\theta})}(t) $$
			is homeomorphic to $\Sigma(\gamma, \beta) \times \mathbb{R}$.
			\item The diameter of $\Sigma(\gamma, \beta)$ is bounded above by $Kd_{H}(\gamma, \beta) + C.$
			\item The set $\Sigma(\theta) = H_{\theta}(0) \cap H_{-\theta}(0)$ is a connected set. 
\item The set $I(\theta) = \tilde{\mathcal{F}}^{s}(\theta) \cap \tilde{\mathcal{F}}^{u}(\theta)$ is a connected set. 
		\end{enumerate}
	\end{lemma}

	\section{$C^{1}$-structural stability implies hyperbolicity of the set of closed orbits of geodesic flows}\label{SECclosedorbits}
	
	The  section is devoted, first of all,  to give some precise definitions of the theory of  hyperbolic dynamics; and secondly, to explain what part of the 
	proof of the stability conjecture for diffeomorphisms  can be extended to geodesic flows. 
	
	We start by recalling some basic definitions concerning hyperbolic dynamics.
	
	\begin{definition} \label{hyperbolic1}
		Let $\psi_{t}: N\longrightarrow N$ be a smooth flow without singularities acting on a complete $C^{\infty}$ Riemannian manifold. An invariant
		set $Y \subset N$ is called hyperbolic if there exists $C>0$, $r>0$, and for every $ p \in Y$ there exist subspaces $E^{s}(p)$, $E^{u}(p)$
		such that
		\begin{enumerate}
			\item $E^{s}(p) \oplus E^{u}(p) \oplus X(p) = T_{p}N$ where $X(p)$ is the subspace tangent to the flow.
			\item $\parallel d_{p}\psi_{t}(v) \parallel \leq Ce^{-rt} \parallel v \parallel $ for every $t \geq 0$ and $v \in E^{s}(p)$.
			\item $\parallel d_{p}\psi_{t}(v) \parallel \leq Ce^{rt} \parallel v \parallel $ for every $t \leq 0$ and $v \in E^{u}(p)$.
		\end{enumerate}
		
		The subspace $E^{s}(p)$ is called stable susbpace, the subspace $E^{u}(p)$ is called the unstable subspace.
		When $Y = N$ the flow $\psi_{t}$ is called Anosov. Replacing $\psi_{t}$ by a diffeomorphism we get what is called an Anosov diffeomorphism.
	\end{definition}
	
	The famous Stable manifold Theorem implies that there exist invariant submanifolds $\bf{W}^{s}(p)$, $\bf{W}^{u}(p)$ for every $p$ in the hyperbolic set  such that 
	$$ \lim_{t \rightarrow +\infty} d(\psi_{t}(x), \psi_{t}(p)) =0$$ 
	for every $x \in \bf{W}^{s}(p)$, and 
	$$ \lim_{t \rightarrow -\infty} d(\psi_{t}(x), \psi_{t}(p)) =0$$ 
	for every $x \in \bf{W}^{u}(p)$. (see for instance \cite{kn:HPS}, \cite{kn:Anosov}). The submanifold $\bf{W}^{s}(p)$ is always tangent to the subspace $E^{s}(p)$ at $p$, the submanifold $\bf{W}^{u}(p)$ is always tangent to the subspace $E^{u}(p)$ at $p$. Item (1) of Lemma \ref{asymptotic-1} yields 
	
	\begin{lemma} \label{Busemann-stable}
		Let $(M,g)$ be a complete Riemannian manifold without conjugate points and let $\Lambda \subset T_{1}M$ be a compact invariant hyperbolic set. Then, for every $\theta \in \Lambda$,  the submanifold $\bf{W}^{s}(\theta)$ is an open (in the relative topology) connected neighborhood of $\theta$ in $\mathcal{F}^{s}(\theta)$. Analogously, $\bf{W}^{u}(\theta)$ is an open (in the relative topology) connected neighborhood of $\theta$ in $\mathcal{F}^{u}(\theta)$
	\end{lemma}
	
	When the geodesic flow of a compact manifold without conjugate points is Anosov the manifold has no conjugate points by a classical result due to Klingenberg \cite{kn:Klingenberg} (see also a nice generalization by R. Ma\~{n}\'{e} \cite{kn:Mane2}). Therefore, Lemma \ref{Busemann-stable} implies that $\bf{W}^{s}(\theta) =\mathcal{F}^{s}(\theta)$, $\bf{W}^{u}(\theta)=\mathcal{F}^{u}(\theta)$ for every $\theta \in T_{1}M$. 
	
	\begin{definition} \label{Ck-stability}
		A smooth flow $\psi_{t} : N\longrightarrow N$ acting on a smooth manifold is $C^{k}$ structurally stable if there exists $\epsilon >0$
		such that every flow $\rho_{t}$ in the $\epsilon$-neighborhood of $\psi_{t}$ in the $C^{k}$ topology is conjugate to $\psi_{t}$. Namely,
		there exists a homeomorphism $h_{\rho}: N\longrightarrow N$ such that
		
		$$h(\psi_{t}(p)) = \rho_{s_{p}(t)}(h(p)) $$
		for every $t \in \mathbb{R}$, where $s_{p}(t)$ is a continuous injective function with $s_{p}(0) =0$.
	\end{definition}
	
	A series of results in the 60's, 70's and 80's characterize $C^{1}$ structurally stable systems (mainly \cite{kn:Robin-Robinson,kn:Mane,kn:Mane88}).
	
	\begin{theorem}
		A diffeomorphism acting on a compact manifold is $C^{1}$ structurally stable if and only if it is Axiom A, namely, the closure of the set of
		periodic orbits is a hyperbolic set and the intersections of stable and unstable submanifolds are always transverse.
	\end{theorem}
	
	This result characterizes as well $C^{1}$ structurally stable flows without singularities on compact manifolds. As for special families of systems, Newhouse \cite{kn:Newhouse} showed that a symplectic diffeomorphism acting on a compact manifold is $C^{1}$-structurally stable if and only if it is Anosov.
	
	To give a context to our results we need to explain in some detail the main ideas of the proof of the so-called stability conjecture: $C^{1}$ structurally stable
	diffeomorphims are Axiom A and invariant submanifolds meet transversally, a result due to Ma\~{n}\'{e} \cite{kn:Mane}.
	
	One of the main steps of Ma\~{n}\'{e}'s proof is the hypebolicity of the closure of the set of periodic orbits of $C^{1}$-structurally stable systems. 
	This step  has been extended and improved for geodesic flows in
	the context of the so-called \textit{Ma\~{n}\'{e} perturbations}. Recall that a $C^{\infty}$ Hamiltonian $H:T^{*}M \longrightarrow \mathbb{R}$
	defined in the cotangent bundle of $M$ is called \emph{Tonelli} if $H$ is strictly convex and superlinear in each tangent space $T_{\theta}T^{*}M$,
	$\theta \in T^{*}M$.
	
	\begin{definition} \label{Mane-peturbations}
		A property $P$ of the Hamiltonian flow of a Tonelli Hamiltonian $H: T^{*}M \longrightarrow \mathbb{R}$ is called $C^{k}$ generic
		from Ma\~{n}\'{e}'s viewpoint if given $\epsilon >0$ there exists a $C^{\infty}$ function $f: M \longrightarrow \mathbb{R}$ whose $C^{k}$ norm is
		less than $\epsilon$ such that the Hamiltonian flow of $H_{f}(q,p) = H(q,p) + f(q)$ has the property $P$. The Hamiltonian $H_{f}$ is called a $C^{k}$ Ma\~{n}\'{e} perturbation of the Hamiltonian $H$.
	\end{definition}
	
	By the Maupertuis principle of classical mechanics (sse \cite{kn:Arnold},  given a Riemannian metric $(M,g)$  every small $C^{k}$ Ma\~{n}\'{e} perturbation $H_{f}$ of the
	Hamiltonian $H(q,p) = \frac{1}{2}g_{q}(p,p)$ defines the Riemannian Hamiltonian of a metric $g_{f}$ that is conformal to $g$.
	
	\begin{definition} \label{stability}
		The geodesic flow $\phi_{t}$ of a compact Riemannian manifold $(M,g)$ is $C^{k}$ structural stable from Ma\~{n}\'{e}'s view point
		if there exists a $C^{k}$ open neighborhood of $g$ such that for each metric in the neighborhood the geodesic flow is conjugate to
		$\phi_{t}$.
	\end{definition}
	
	The assumption of structural stability from Ma\~{n}\'{e}'s viewpoint is weaker than the usual one, since it requires persistent dynamics in a neighborhood of
	special type of perturbations of the metric, the conformal ones. Clearly, $C^{1}$ structural stability implies $C^{2}$ structural stability from Ma\~{n}\'{e}'s viewpoint (this is why some authors call this notion $C^2$-structural stability, see for instance \cite{kn:CM}). The next result will be crucial for the proof of the main Theorem of the article. 
	
	\begin{theorem} \label{persistence} (see \cite{kn:ARR})
		Let $(M,g)$ be a compact manifold whose geodesic flow is $C^{2}$-structurally stable from Ma\~{n}\'{e}'s viewpoint.
		Then the closure of the set of periodic orbits is a hyperbolic set for the geodesic flow.
	\end{theorem}
	
	The second part of Ma\~{n}\'{e}'s proof of the $C^{1}$-stability conjecture for diffeomorphisms relies on proving that the non-wandering set 
	is the closure of the set of periodic orbits. To do this,  Pugh's $C^{1}$-closing lemma is essential, so this step of Ma\~{n}\'{e}'s work does not extend 
	to  geodesic flows (except on surfaces, cf. \cite{kn:Irie}).  
	
	\section{Regularity of Horospherical foliations and foliated open neighborhoods}
	
	The purpose of the subsection is to explore in more detail some consequences of Theorem \ref{fol-continuity} to show the existence of some special open, $s, u$-foliated sets that will be relevant in the sequel. Our main references are the works of Pesin \cite{kn:Pesin} and Eberlein \cite{kn:Eberlein-96} about mild regularity properties of horospheres. The motivation for the results of the section come from the theory of topological dynamics: foliated neighborhoods appear in Anosov dynamics where we have local product structure for the invariant foliations $\mathcal{F}^{s}$, $\mathcal{F}^{u}$. Without hyperbolicity, we loose the transversality of these foliations, so it is not clear at all that we might get open, foliated neighborhoods. What we show in the section is that the vertical fibers can play the role of "cross sections" for the foliations $\mathcal{F}^{s}$ and $\mathcal{F}^{u}$, providing the existence of a family of foliated open neighborhoods that are somehow fibrations with base at open subsets of the vertical fibers. 
	
	A survey of the basic regularity properties of horospheres requires the introduction of some notations concerning the geometry of the unit tangent bundle. We refer to  \cite{kn:Paternain}, \cite{kn:Ruggiero-Ensaios} for details. Let us denote by $H(\theta) \subset T_{\theta}T_{1}\tilde{M}$ the horizontal subspace, $\mathcal{V}(\theta)\subset T_{\theta}T_{1}\tilde{M}$ the vertical subspace, and let $V(\theta) \subset T_{1}\tilde{M}$ be the vertical fiber of $\theta$. We have that  $T_{\theta}T_{1}\tilde{M}= H(\theta)\oplus \mathcal{V}(\theta)$ and that these subspaces are orthogonal with respect to the Sasaki metric. Moreover, if $\mathcal{X}(\theta)$ is the unit vector tangent to the geodesic flow, and $\mathcal{H}(\theta) \subset H(\theta)$ is the subspace orthogonal to $\mathcal{X}(\theta)$ with respect to the Sasaki metric, then $T_{\theta}T_{1}\tilde{M}= \mathcal{H}(\theta)\oplus \mathcal{V}(\theta)\oplus <\mathcal{X}(\theta)>$ is an orthogonal decomposition. We shall denote by $N(\theta)$ the subspace $N(\theta) =  \mathcal{H}(\theta)\oplus \mathcal{V}(\theta)$. 
	
	Pesin in \cite{kn:Pesin} shows that given $\theta \in T_{1}\tilde{M}$, the second fundamental form of the sphere $S_{T-t}(\gamma_{\theta}(T))$, for $T >t$, centered at $\gamma_{\theta}(T)$ with radius $T-t$ at the point $p= \gamma_{\theta}(t)$ is a linear operator 
	$$U_{T}(t) : T_{\gamma_{\theta}(T-t)} S_{T}(\gamma_{\theta}(T)) \longrightarrow T_{\gamma_{\theta}(T-t)} S_{T}(\gamma_{\theta}(T)) $$
	satisfying the matrix (with respect to an orthonormal frame $\{e_{i}(t)\}$ along $\gamma_{\theta}$ where $e_{1}(t) = \gamma_{\theta}'(t)$ is one of the vectors in the frame) Riccati equation 
	$$ U'(t) + U^{2}(t) + K(t)=0$$ 
	where the derivative is the covariant derivative along $\gamma_{\theta}$, and $K(t)$ is a matrix whose entries are $K_{ij}(t) = g(\mathcal{R}(e_{1},e_{i})e_{1},e_{j})$, $\mathcal{R}$ being the curvature tensor. This implies that the tangent space of the submanifold $\tilde{\mathcal{F}}^{s}_{T}(\theta) $  at the point $\theta$ is given by the subspace 
	$$T_{\theta} \tilde{\mathcal{F}}^{s}_{T}= \{ (W, U_{T}(0)(W)) , \mbox{ } W \in \mathcal{H}(\theta) \}, $$
	in the orthogonal decomposition $ N(\theta) = \mathcal{H}(\theta)\oplus \mathcal{V}(\theta)$. 
	By the comparison theory of the Riccati equation, there exists $L>0$ depending on lower bounds of the sectional curvatures of $(M,g)$, and $T>0$ that might depend on $\theta$, such that $\parallel U_{T}(0) \parallel \leq L$. So the angle formed by $\mathcal{V}_{\theta}$ and the tangent spaces of the submanifolds $\tilde{\mathcal{F}}^{s}_{T}(\theta)$ at $\theta$ is bounded above by some positive constant $\alpha >0$ depending on $L$ and $T$.  
	
	In particular, for large $T$ the subspace $T_{\theta} \tilde{\mathcal{F}}^{s}_{T}(\theta)$ is transverse to the subspace $ \mathcal{V}(\theta)\oplus <\mathcal{X}(\theta)>$, forming an angle that is bounded from below by a positive constant that is uniform in $T_{1}\tilde{M}$.

	Let $V_{\delta}(\theta) \subset V(\theta)$ be a small ball of radius $\delta$ in $V(\theta)$ centered at $\theta$, with respect to the Sasaki metric. Notice that the tangent bundle of the submanifold 
	$$ Q_{\delta, \sigma}(\theta) = \cup_{\mid t \mid <\sigma}\phi_{t}(V_{\delta}(\theta)) $$ 
	is contained in a small cone around the subbundle of subspaces $ \mathcal{V}(\eta)\oplus <\mathcal{X}(\eta)>$ for $\eta \in  Q_{\delta, \sigma}(\theta)$. 
		
	\begin{lemma} \label{vertical-neigh}
		Let $(M,g)$ be a compact Riemannian manifold without conjugate points such that geodesic rays diverge in $(\tilde{M}, \tilde{g})$. Let $ U_{\theta}^{s} \subset \tilde{\mathcal{F}}^{s}(\theta)$ be a connected, open coordinate subset of the submanifold, containing $\theta$ with compact closure.  Then the set  
		$$ Q_{ \delta, \sigma , U_{\theta}}^{s}(\theta) = \cup_{\xi \in U_{\theta}}Q_{\delta, \sigma}(\xi)= \phi_{t}(\cup_{\mid t \mid <\sigma , \xi \in U_{\theta}}V_{\delta}(\xi))$$
		is an open neighborhood of $\theta$ in $T_{1}\tilde{M}$. 
	\end{lemma}
	
	\begin{proof}
		The proof is a consequence of Corollary \ref{fol-continuity} and the general theory of the Riccati equation according to the remarks before Lemma \ref{vertical-neigh}. Let us sketch the proof for the sake of completeness. 
		
		Since the submanifolds $\tilde{\mathcal{F}}^{s}_{T}(\theta) $ converge uniformly in compact sets to $\tilde{\mathcal{F}}^{s}(\theta)$, and their tangent bundles are transverse and uniformly bounded away from  the tangent bundle of the submanifolds $Q_{\delta, \sigma}(\theta)$, then the transverse intersection between $\mathcal{F}^{s}_{T}(\theta) $ and  $Q_{\delta, \sigma}(\xi) $ for $\xi$ in an open neighborhood of $\theta$ will nonempty for $T$ large. Since the closure of $U_{\theta}^{s}$ is compact, and $\tilde{\mathcal{F}}^{s}_{T}(\theta) $ converges uniformly on compact subsets to $\tilde{\mathcal{F}}^{s}(\theta)$, there exists $T_{0}$ large such that  
		$\tilde{\mathcal{F}}^{s}_{T}(\theta) $ will cross transversally $Q_{\delta, \sigma}(\xi) $ for every $\xi \in U_{\theta}$ and $T >T_{0}$. This implies that the limit submanifold $\tilde{\mathcal{F}}^{s}(\theta)$ will cross each $Q_{\delta, \sigma}(\xi) $ at just one point $\xi \in U_{\theta}$ (and actually at every $\xi \in \tilde{\mathcal{F}}^{s}$. 
		
		Now, we can parametrize the set $Q_{ \delta, \sigma , U_{\theta}}^{s}(\theta)$ using the parametrizations of $U_{\theta}$, $V_{\delta}(\theta)$, and the parameter $t$, to get a homeomorphism from $Q_{ \delta, \sigma , U_{\theta}}(\theta) $ to an open subset of $\mathbb{R}^{2n-1}$. Thus, the invariance of domain Theorem grants that $Q_{ \delta, \sigma , U_{\theta}}(\theta) $ is an open subset. 
		
	\end{proof}

	\begin{lemma} \label{Fol-neigh}
		Let $(M,g)$ be a compact manifold without conjugate points such that geodesic rays diverge in $(\tilde{M}, \tilde{g})$. Then, given $\theta \in T_{1}\tilde{M}$, $U_{\theta}$ as above, there exists $\delta ' >0$ such that if $d(\theta, \eta) < \delta '$ we have
		\begin{enumerate}
			\item Each submanifold $\tilde{\mathcal{F}}^{s}(\eta)$  crosses each set $Q_{\delta, \sigma}(\xi)$ for every 
			$\xi \in U_{\theta}^{s}$ at just one point $\eta^{s}(\xi)$. 
			\item The sets 
			$$ \tilde{\mathcal{F}}^{s}_{U_{\theta}}(\eta) = \tilde{\mathcal{F}}^{s}(\eta)\cap Q_{ \delta, \sigma , U_{\theta}^{s}}^{s}(\theta)$$
			are all homeomorphic to $U_{\theta}$. 
			\item The sets 
			$$ S_{U_{\theta}, \delta '}^{s} = \cup_{\eta \in V_{\delta '}(\theta) } \tilde{\mathcal{F}}^{s}_{U_{\theta}^{s}}(\eta) $$ 
			are s-foliated, continuous cross sections for the geodesic flow homeomorphic to $U_{\theta} \times  V_{\delta '}(\theta)$, and the sets 
			$$\Gamma_{U_{\theta}^{s}, \delta ', \sigma}^{s} = \cup_{\mid t \mid < \sigma} \phi_{t}( S_{U_{\theta}, \delta '}^{s} ) $$
			are s-foliated open neighborhoods of $\theta$.
		\end{enumerate}
		Similar statements hold to u-foliated sections and neighborhoods. 
	\end{lemma}
	
	\begin{proof}
		
		Item (1) follows from Corollary \ref{fol-continuity} and the fact that the tangent space of large spheres at each point is contained in a cone that is uniformly away from the vertical subspace. 
		
		Item (2) follows from item (1), since the intersection of each leaf $\tilde{\mathcal{F}}^{s}(\eta)$ with each submanifold $Q_{\delta, \sigma}(\xi)$ for $\xi \in U_{\theta}^{s}$ gives rise to a homeomorphism between $U_{\theta}^{s}$ and $\mathcal{F}^{s}(\eta) \cap Q_{ \delta, \sigma , U_{\theta}^{s}}(\theta)$. Namely, to each $\xi \in U_{\theta}^{s}$ associate the point 
		$$f_{\eta}(\xi) = \tilde{\mathcal{F}}^{s}(\eta) \cap Q_{ \delta, \sigma} (\xi) . $$
		This map is obviously continuous with continuous inverse. 
		
		Item (3) is straightforward from item (2) and the invariance of domain Theorem.

	\end{proof}
	
	\section{The existence of local product neighborhoods around hyperbolic points}
	
	In this section we deal with the problem of the existence of local product neighborhoods of points despite the fact that the intersections between the leaves in $\mathcal{F}^{s}$ and $\mathcal{F}^{u}$ are not in general transverse, the set $I(\theta)$ defined in Corollary \ref{strip} might not be just one point. 
	
	Through the section, we shall say that a point $\theta \in T_{1}\tilde{M}$ is a \emph{hyperbolic point} if the leaves $\tilde{\mathcal{F}}^{s}(\theta)$ and $\tilde{\mathcal{F}}^{u}(\theta)$ are smooth and transverse at $\theta$. 
	
	According to Theorem \ref{persistence}, every lift in $T_{1}\tilde{M}$ of a point $\eta$ in the closure of the set of periodic points of the geodesic flow in $T_{1}M$ is a hyperbolic point, since we know that the invariant submanifolds are just the projections of the leaves $\tilde{\mathcal{F}}^{s}(\tilde{\eta})$, $\tilde{\mathcal{F}}^{u}(\tilde{\eta})$, by the covering map $P : T_{1}\tilde{M}\longrightarrow T_{1}M$, where $P(\tilde{\eta}) = \eta$.

	\begin{definition} \label{local-product-0}
		Let $(M,g)$ be a compact Riemannian manifold without conjugate points. We say that an open  neighborhood $U$ of a point $\theta \in T_{1}\tilde{M}$ is a local product neighborhood if for every pair of points $\xi , \eta \in U$ we have that 
		$$ \tilde{\mathcal{F}}^{s}(\xi) \cap \tilde{\mathcal{F}}^{cu}(\eta) \neq \emptyset $$
		$$\tilde{ \mathcal{F}}^{u}(\xi) \cap \tilde{\mathcal{F}}^{cs}(\eta) \neq \emptyset $$
	\end{definition}
	
	The goal of the section is to show that
	
	\begin{theorem} \label{local-product}
		Let $(M,g)$ be a compact Riemannian manifold without conjugate points such that $(\tilde{M}, \tilde{g})$ is a quasi-convex space and geodesic rays diverge. Then every hyperbolic point in $T_{1}\tilde{M}$ has an open local product neighborhood. 
	\end{theorem}
	\bigskip
	
	Since the covering map $P : T_{1}\tilde{M}\longrightarrow T_{1}M$ is a local diffeomorphism, the image by $P$ of a small local product neighborhood of a hyperbolic point $\tilde{\eta} \in T_{1}\tilde{M}$ provides a local product neighborhood for $\eta = P(\tilde{\eta})$. This fact will be very important for the proof of the main Theorem.  
	
	\subsection{Semi-continuity of bi-asymptotic classes}
	
	\bigskip
	
	\begin{definition} \label{bi-asympt-class}
		Let $\theta \in T_{1}\tilde{M}$, and $\Sigma(\theta)= H_{\theta}(0) \cap H_{-\theta}(0)$ be the set defined in Lemma \ref{strip}. We say that the set 
		$$ I(\theta) = \{ (q, -\nabla_{q}b^{\theta}) , \mbox{ } q \in \Sigma(\theta) \}= \tilde{\mathcal{F}}^{s}(\theta) \cap \tilde{\mathcal{F}}^{u}(\theta) $$ 
		is the bi-asymptotic class of $\theta$. 
	\end{definition}
	
	By Lemma \ref{strip}, if $\theta$ is a hyperbolic point then $I(\theta) = \theta$. Moreover, $I(\theta) = \theta$ if and only if the leaves $\tilde{\mathcal{F}}^{s}(\theta)$ and $\tilde{\mathcal{F}}^{u}(\theta)$ meet just at $\theta$. Such kind of points are called  generalized rank one points in \cite{kn:ARR}. 
	
	The following statement describes a kind of semi-continuity of the bi-asymptotic classes that is similar to the semi-continuity of bi-asymptotic classes in nonpositive curvature (see also \cite{kn:MR}). 
	
	\begin{lemma} \label{semi-cont}
		Given $\theta \in T_{1}\tilde{M}$ such that $I(\theta)$ is a compact set, and $\epsilon >0$, there exists $\delta >0$ such that if $d(\eta, \theta) < \delta$ then $I(\eta)$ is contained in the $\epsilon$-tubular neighborhood $U_{\epsilon}(I(\theta))$ of $I(\theta)$ in $T_{1}\tilde{M}$. 
	\end{lemma}
	
	\begin{proof}
		
		Suppose by contradiction that the statement is not true. Then there exists $\epsilon_{0} >0$ such that for every $n\in \mathbb{N}$, there exists $\eta_{n} \in T_{1}\tilde{M}$ with $d(\eta_{n}, \theta) < \frac{1}{n}$ and $I(\eta_{n})$ is not contained in $U_{\epsilon_{0}}(I(\theta))$. Since $I(\eta_{n})$ is connected, and each connected component of the boundary $\partial U_{\epsilon_{0}}(I(\theta))$ of $U_{\epsilon_{0}}(I(\theta))$ disconnects the ambient space, there must be some point $\bar{\eta}_{n} \in I(\eta_{n}) \cap \partial U_{\epsilon_{0}}(I(\theta))$. By the generalized strip Theorem (Lemma \ref{strip}), we have that 
		\begin{eqnarray*}
			d_{H}(\gamma_{\bar{\eta}_{n}}, \gamma_{\eta_{n}}) & \leq &  Kd(\gamma_{\bar{\eta}_{n}}(0), \gamma_{\eta_{n}}(0)) +C  \\
			&\leq & Kd(\bar{\eta}_{n}, \eta_{n}) +C \\
			&\leq & K(d(\bar{\eta}_{n}, \theta) + d(\theta, \eta_{n})) + C \leq K(\epsilon_{0} + \frac{1}{n}) + C
		\end{eqnarray*}
		Then, taking convergent subsequences of $\eta_{n}$, $\bar{\eta}_{n}$, we get limit points $\eta_{\infty} \in I(\theta)$, $\bar{\eta}_{\infty} \in \partial U_{\epsilon_{0}}(I(\theta))$ such that 
		$$ d_{H}(\gamma_{\bar{\eta}_{\infty}}, \gamma_{\eta_{\infty}})  \leq   K\epsilon_{0} + C $$
		Since the geodesic $\gamma_{\eta_{\infty}}$ is bi-asynptotic to $\gamma_{\theta}$, the geodesic 
		$\gamma_{\bar{\eta}_{\infty}}$ is bi-asymptotic to $\gamma_{\theta}$ as well . Therefore, $\bar{\eta}_{\infty} \in I(\theta)$ which is a contradiction since the set $\partial U_{\epsilon_{0}}(I(\theta))$ is disjoint from $U_{\epsilon_{0}}(I(\theta))$.

	\end{proof}
	
	\begin{corollary} \label{rank-one-cont}
		The map $\theta \rightarrow I(\theta)$ is continuous in the Hausdorff topology if $I(\theta)= \theta$. In particular, every hyperbolic point is a point of continuity of the map $\theta \rightarrow I(\theta)$. 
	\end{corollary}
	
	\subsection{Local product structure near hyperbolic points through  differential topology}

	\bigskip
	
	Let $\mathbb{D}^k_a = \{ x \in \RR^k \ : \ \|x\|\leq a \}$ and $\mathbb{D}^k = \mathbb{D}^k_1$. Denote by 
	$$S^k_a =\{ x \in \RR^k \ : \ \|x\|=a\} = \partial \mathbb{D}^k_a$$ and let $S^k= S^k_1$. 
	
	Let $\varphi: \mathbb{D}^k \to \RR^d$ be a Lipschitz embedding. Let us define, for $ x \in \mathrm{int}(\mathbb{D}^k)$ the set $\partial \varphi (x)$ to be the set of tangent vectors $v \in \RR^d$ in the following sense: there is a curve $\gamma: (-\eps, \eps) \to \mathbb{D}^k$ with $\gamma(0)=x$ so that for some $t_k \to 0$ we have that $|\varphi \circ \gamma(t_k) - \varphi(x) - t_k v| = o(t_k)$. Note that being a Lipschitz embedding implies in particular (by definition) that $\partial \varphi(x)$ contains a $k$-dimensional subspace and is contained in a $k$-dimensional cone around such a subspace. 
	
	We say that two embeddings $\varphi_1: \mathbb{D}^{\ell_1}  \to \RR^d$,  $\varphi_2: \mathbb{D}^{\ell_2} \to \RR^d$ with $\ell_1 + \ell_2 =d$  and   $\varphi_1(0)=\varphi_2(0)=0$, are \emph{transverse} at $0$ if $\partial \varphi_1 (0) \cap \partial \varphi_2(0) = \{0\}$.
	
	Note that up to change of coordinates, if $\varphi_1: \mathbb{D}^{\ell_1}  \to \RR^d$ and $\varphi_2: \mathbb{D}^{\ell_2} \to \RR^d$ with $\ell_1 + \ell_2 =d$ are transverse at $\varphi_1(0)=\varphi_2(0)=0$, we can assume without loss of generality that 
	there are uniformly Lipschitz functions $\tilde{\varphi}_1: \mathbb{D}^{\ell_1} \to \RR^{\ell_2}$ and  $\tilde{\varphi}_2: \mathbb{D}^{\ell_2} \to \RR^{\ell_1}$ so that $\varphi_1(x) = (x, \tilde{\varphi}_1(x))$ and $\varphi_2(y)=(\tilde{\varphi}_2(y),y)$.  Moreover, by a further change of coordinates we can assume that the functions $\tilde{\varphi}_i$ both verify that for $x \in \mathbb{D}^{\ell_i}$ small:
	
	\begin{equation}\label{eq:110lip}
		\| \tilde \varphi_i(x) \| \leq \frac{1}{10} \|x\|. 
	\end{equation}

	Given $\delta>0$ we say $\psi_1, \psi_2$ are $\delta$-\emph{perturbations} of $\varphi_1, \varphi_2$ if there they are $\delta$-$C^0$-close (that is,  $\|\psi_i(x) - \varphi_i(x)\|< \delta$ for every $x \in \mathbb{D}^{\ell_i}$) and the images of $\psi_i$ can be written as a graph on the previously chosen coordinates. 
	
	Now we can prove: 
	
	\begin{lema} \label{local-prod}
		Let $\varphi_1: \mathbb{D}^{\ell_1}  \to \RR^d$ and $\varphi_2: \mathbb{D}^{\ell_2} \to \RR^d$ with $\ell_1 + \ell_2 =d$ be Lipschitz embeddings so that $\varphi_1(0) = \varphi_2(0)= 0$  which are transverse at $0$. Then, for every $\eps>0$ there exists $\delta>0$ such that if $\psi_1: \mathbb{D}^{\ell_1}  \to \RR^d$ and $\psi_2: \mathbb{D}^{\ell_2} \to \RR^d$ are $\delta$-perturbations, then we have that $\psi_1(\mathbb{D}^{\ell_1})$ intersects $\psi_2(\mathbb{D}^{\ell_2})$ at  a point within a distance $\eps$ from the origin $0 \in \RR^d$.  
	\end{lema}
	
	\begin{proof}
		Fix $\eps>0$. Given $a>0$ small, where equation \eqref{eq:110lip} holds, we know that $\|\tilde \varphi_i(y)\| \leq a/10$ for all $y \in \mathbb{D}^{\ell_i}_a$. We can assume that $a\ll \eps$. 
		
		The images of $S^{\ell_i}_a$ by $\varphi_i$ are disjoint spheres of the form $$S_1 = \{(x, \tilde \varphi_1(x)) \ : \ \|x\|=a \}$$ and $$S_2 = \{ (\tilde \varphi_2(y), y) \ : \ \|y\|=a\}.$$ 
		
		Now, we consider $\delta \ll a/100$ and let $\psi_i$ be $\delta$-perturbations of $\varphi_i$. By making the change of variables $(x,y) \mapsto (x, y - \tilde \psi_1(x))$ we can assume that $\psi_i$ is the inclusion of $\mathbb{D}^{\ell_1}_a$ on $\mathbb{D}^{\ell_1} \times \{0\}$ (equivalently that $\tilde \varphi_1=0$. Note that since $\delta \ll a/100$ the image of $\psi_2$ of $S^{\ell_2}_a$ is still $a/50$-close to $S_2$ whose projection to $\{0\} \times ( \RR^{\ell_2} \setminus \{0\})$ is close to the inclusion of $S^{\ell_2}_a$ in $\{0\} \times \RR^{\ell_2}$. Moreover, the image of $\psi_2$ of $\DD^{\ell_2}_a$ is $a/2$ close to $\{0\} \times \RR^{\ell_2}$ it follows that the projection of the image of $\psi_2$ to $\{0\} \times \RR^{\ell_2}$  must\footnote{This can be shown with a degree argument, as the projection of the image of $\psi_2$ of the boundary of $\DD^{\ell_2}_a$ is close to the inclusion of the sphere $S^{\ell_2}_a$ after normalizing the points of this inclusion we get a degree one map from the sphere to the sphere. Since this map can be filled by the projection of the image of $\DD^{\ell_2}_a$ it follows that this image must cover the whole interior of $S^{\ell_2}_a$ in particular contain $0$.} intersect $(0,0)$ and therefore must intersect the image of $\psi_1$ completing the proof. 
	\end{proof}

	\bigskip
	
	\textbf{Proof of Theorem \ref{local-product}}
	\bigskip
	
	The proof of Theorem \ref{local-product} follows from the application of Lemma \ref{local-prod} to the submanifolds $\tilde{\mathcal{F}}^{s}(\theta) $ and $\tilde{\mathcal{F}}^{cu}(\theta)$ at a hyperbolic point $\theta \in T_{1}\tilde{M}$. Indeed, we have that they have complementary dimensions in $T_{1}\tilde{M}$, $(n-1)$ for the first one and $n= dim(M)$ for the second one, and that they are transverse at $\theta$ by assumption. Moreover, the foliations $\tilde{\mathcal{F}}^{s}$, $\tilde{\mathcal{F}}^{cu}$ are continuous by the assumptions of Theorem \ref{local-product}. So up to a local coordinate change in a small neighborhood of $\theta$, the leaves $\tilde{\mathcal{F}}^{s}(\theta) $ and $\tilde{\mathcal{F}}^{cu}(\theta)$ locally satisfy the assumptions of the embeddings $\varphi_{1}$, $\varphi_{2}$, and nearby leaves satisfy the assumptions of the  perturbations $\tilde{\varphi}_{1}$, $\tilde{\varphi}_{2}$ defined in Lemma \ref{local-prod}. An analogous reasoning shows that the leaves 
	$\tilde{\mathcal{F}}^{cs}$, $\tilde{\mathcal{F}}^{u}$ intersect locally 
	at every point in an open neighborhood of $\theta$. This finishes the proof of Theorem \ref{local-product}. 
	
\bigskip

{\textbf Remark:}

Theorem \ref{local-product} implies in some sense that in a small neighborhood of a hyperbolic point there is a sort of weak hyperbolicity. We know that the sets $I(\eta)$ where the transverse intersection between invariant leaves fails is a small set according to Lemma \ref{semi-cont}, and that stable and unstable invariant leaves cross at every point in the neighborhood. The geodesic flow is like an expansive flow up to the identification of each set $I(\eta)$ with a single point by an equivalence relation (see for instance \cite{kn:GR1,kn:GR2}).

	\section{Contraction of a stable leaf outside a bi-asymptotic class under the action of the geodesic flow}
	
	The goal of the section is to study weak hyperbolic properties of the dynamics in open neighborhoods of a bi-asymptotic class $I(\theta)$ for $\theta \in T_{1}\tilde{M}$. Recall that $P : T_{1}\tilde{M} \longrightarrow T_{1}M$ is the covering map $P (p,v) = (\pi(p), d\pi(v))$, induced by the covering map $\pi$. 
	
	Let us recall some objects and notations of Section 4. For $V_{\delta}(\theta) \subset V(\theta)$ a small ball of radius $\delta$ in the vertical fiber $V(\theta)$ centered at $\theta$, with respect to the Sasaki metric, we defined the sets 
	$$ Q_{\delta, \sigma}(\theta) = \cup_{\mid t \mid <\sigma}\phi_{t}(V_{\delta}(\theta)), $$ 
	$$ Q_{ \delta, \sigma , U_{\theta}}(\theta) = \cup_{\xi \in U_{\theta}}Q_{\delta, \sigma}(\xi)= \phi_{t}(\cup_{\mid t \mid <\sigma , \xi \in U_{\theta}}V_{\delta}(\xi)) , $$
	and 
	$$ S_{U_{\theta}, \delta '}^{s} = \cup_{\eta \in V_{\delta '}(\theta) } \tilde{\mathcal{F}}^{s}_{U_{\theta}}(\eta) .$$ 
	Here, $U_\theta$ is an open coordinate neighborhood of $\theta$ in $\tilde{\mathcal{F}}^{s}(\theta)$, and 
	$\tilde{\mathcal{F}}^{s}_{U_{\theta}}(\eta) $ is an relative open set in $\tilde{\mathcal{F}}^{s}(\eta)$ that is homeomorphic to $U_\theta$.

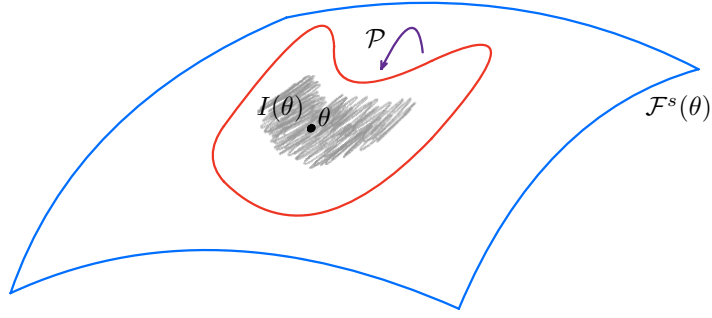
\begin{figure}[ht]
\begin{center}
\includegraphics[scale=0.84]{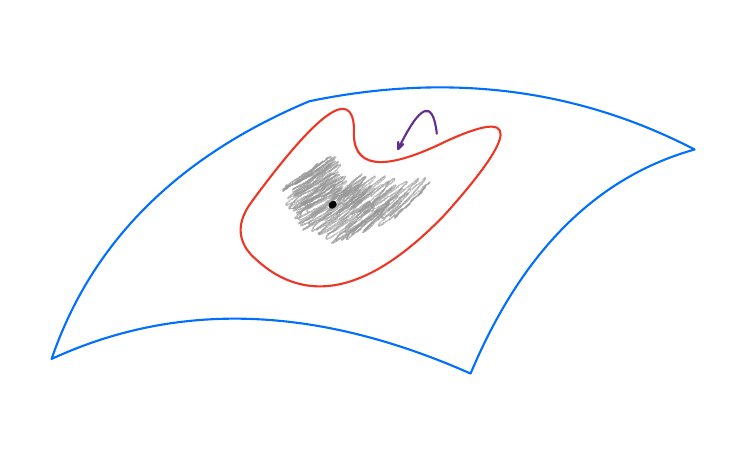}
\begin{picture}(0,0)
\put(-189,104){$I(\theta)$} 
\put(-166,100){$\theta$} 
\put(-42,104){$\mathcal{F}^s(\theta)$}
\put(-148,130){$\mathcal{P}$}
\end{picture}
\end{center}
\vspace{-0.5cm}
\caption{{\small The first return map $\mathcal{P}$ maps $U_\theta \cap \mathcal{F}^s(\theta)$ to its interior containing $I(\theta)=\mathcal{F}^s(\theta)\cap \mathcal{F}^u(\theta)$ the biasymptotic class.}}\label{fig-1}
\end{figure}

	According to Lemma \ref{vertical-neigh}, the second set is an open set in the unit tangent bundle, and according to Lemma \ref{Fol-neigh}, the third set is a $s$-foliated  continuous cross section for the geodesic flow homeomorphic to $U_{\theta} \times  V_{\delta '}(\theta)$.
	
	Moreover, by Lemma \ref{Fol-neigh}, each open relative set $\tilde{\mathcal{F}}^{s}_{U_{\theta}}(\eta)$ for $\eta \in V_{\delta '}(\theta)$ intersects $Q_{\delta, \sigma}(\xi)$ transversally at just one point for each $\xi \in U_{\theta}$. 
	
	The transversal intersections $\tilde{\mathcal{F}}^{s}_{U_{\theta}}(\eta)\cap Q_{\delta, \sigma}(\xi)$ induces a projection map
	$$ \tilde{P}^{h}: S_{U_{\theta}, \delta '}^{s} \longrightarrow U_{\theta}$$ 
	given by 
	$$\tilde{P}^{h}(\omega) = \xi (\omega) $$
	where $\xi(\omega) \in U_{\theta}$ is such that $\omega \in Q_{\delta, \sigma}(\xi (\omega))$. This map is actually one of the components of the homeomorphism
	$$f_{\theta} : S_{U_{\theta}, \delta '}^{s}\longrightarrow U_{\theta} \times  V_{\delta '}(\theta)$$
	described in Lemma \ref{Fol-neigh}. Indeed, if $\eta(\omega) \in V_{\delta '}(\theta)$ is such that $\omega \in \tilde{\mathcal{F}}^{s}_{U_{\theta}}(\eta(\omega))$, then 
	$$f_{\theta}(\omega) = (\xi(\omega), \eta(\omega))$$
	are in some sense coordinates in the product $U_{\theta} \times  V_{\delta '}(\theta)$ of a point $\omega \in S_{U_{\theta}, \delta '}^{s}$.

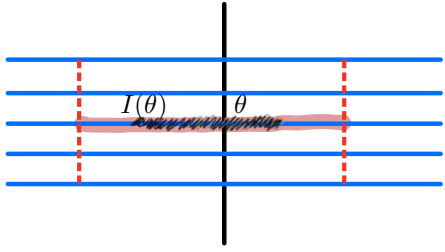
\begin{figure}[ht]
\begin{center}
\includegraphics[scale=0.84]{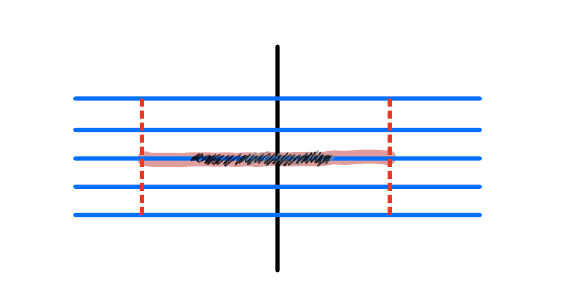}
\begin{picture}(0,0)
\put(-159,64){$I(\theta)$} 
\put(-116,64){$\theta$} 
\end{picture}
\end{center}
\vspace{-0.5cm}
\caption{{\small A schematic drawing of the region $S^s_{U_{\theta,\delta'}}$ which is the region between the dotted lines. The intersection with each leaf of $\mathcal{F}^s$ (the horizontals) through $\eta$ is the set $U_\theta(\eta)$.}}\label{fig-2}
\end{figure}

	Moreover, the restriction of $\tilde{P}^{h}$ to $\tilde{\mathcal{F}}^{s}_{U_{\theta}}(\eta)$ gives the homeomorphism between this set and $U_{\theta} \subset \tilde{\mathcal{F}}^{s}(\theta)$ considered in Lemma \ref{Fol-neigh}. 
	
	Similar objects can be defined in an open neighborhood of $\Pi(\theta) \in T_{1}M$, because $P$ is a covering map. So let $U_{\Pi(\theta)} = P(U_{\theta}) \subset \mathcal{F}^{s}(P(\theta))$, $\mathcal{F}^{s}_{U_{P(\theta)}}(P(\eta)) = P(\tilde{\mathcal{F}}^{s}_{U_{\theta}}(\eta))$, let $\it{S}_ {U_{P(\theta)}, \delta '}^{s}$ be the section $P(S_{U_{\theta}, \delta '}^{s})$, and let 
	$$ \it{P}^{h} : \it{S}_ {U_{P(\theta)}, \delta '}^{s} \longrightarrow U_{P(\theta)}$$
	be the projection induced by $\tilde{P}^{h}$. 
	
	The main result of the section is the following:
	
	\begin{proposition} \label{contraction}
		Let $(M,g)$ be a compact Riemannian manifold without conjugate points such that $(\tilde{M}, \tilde{g})$ is a quasi-convex space where geodesic rays diverge. Let $\theta \in T_{1}\tilde{M}$ such that $I(\theta)$ is contained in an open coordinate neighborhood $U_{\theta}$ of $\theta$ in $\tilde{\mathcal{F}}^{s}(\theta)$. Suppose that $P(\theta)$ is a recurrent point. Let 
		$$\mathcal{P}:  \it{S}_ {U_{\Pi(\theta)}, \delta '}^{s}\longrightarrow \it{S}_ {U_{P(\theta)}, \delta '}^{s}$$ be the Poincar\'{e} first return map of the section $\it{S}_ {U_{P(\theta)}, \delta '}^{s}$, and suppose that  
		$$\lim_{k\rightarrow \infty}\mathcal{P}^{n_{k}}(P(\theta)) = P(\theta) .$$ 
		Let $\epsilon >0$ small and let $K \subset U_{\theta}$ be a compact set containing an $\epsilon$- tubular neighborhood $U_{\epsilon}(P(I(\theta))$ of $P(\theta)$. Then there exists $n_{\theta}>0$ such that $\it{P}^{h}(\mathcal{P}^{n_{k}}(K)) \subset U_{\epsilon}(P(I(\theta))$ for every $n_{k} \geq n_{\theta}$. 
	\end{proposition}
	
	\begin{proof}
		
		The proof is by contradiction. If the statement of the proposition is not true, then we have two possibilities:
		\begin{enumerate}
			\item There exist a sequence $\omega_{i} \in K$ with $k_{i} \rightarrow +\infty$, and $m_{i} \rightarrow +\infty$,  such that the Poincar\'{e} map $\mathcal{P}^{n_{k_{i}}}$ is not defined in the forward orbit $\phi_{t}(\omega_{i})$ for $t \geq m_{i}$, 
			\item For every point $\omega \in K$ there are infinite elements of the forward orbit $\{ \phi_{t_{i(\omega)}}(\omega) \}$ such that 
			$$\mathcal{P}^{n_{k_{i(\omega)}}}(\omega) = \phi_{t_{i(\omega)}}(\omega)$$ 
			but 
			$$ \it{P}^{h}(\mathcal{P}^{n_{k_{i(\omega)}}}(\omega) ) \notin  U_{\epsilon}(P(I(\theta)). $$ 
		\end{enumerate}
		
		In any case, there exists a sequence $\tilde{\xi}_{k} \in T_{1}\tilde{M}$ of lifts of points $\xi_{k} \in K$,  such that 
		\begin{enumerate}
			\item $\tilde{\xi_{i}} \in U_{\theta} \subset \tilde{\mathcal{F}}^{s}(\theta)$, 
			\item $\phi_{t_{i}}(\tilde{\xi}_{k}) \in \phi_{t_{i}}(\tilde{\mathcal{F}}^{s}(\theta) )= \tilde{\mathcal{F}}^{s}(\phi_{t_{i}}(\theta) )$. 
			\item $P(\phi_{t_{i}}(\theta) ) = \mathcal{P}^{n_{k_{i}}}(P(\theta) )$ where $n_{k_{i}} \rightarrow +\infty$. 
			\item $\phi_{t_{i}}(\tilde{\xi}_{i}) \notin \tilde{\mathcal{F}}^{s}_{U_{\theta}}(\eta_{i})$ 
			where $\eta_{i} = \tilde{\mathcal{F}}^{s}(\phi_{t_{i}}(\theta) ) \cap V_{\delta '}(\theta)$.
		\end{enumerate}
		
		Let $\tau_{i} = \phi_{t_{i}}(\tilde{\xi}_{i})$. By Lemma \ref{asymptotic-1} item (3), we have that 
		$$ d(\gamma_{\phi_{t_{i}}(\theta)}(t) , \gamma_{\tau_{i}}(t)) \leq K d(\gamma_{\phi_{t_{i}}(\theta)}(0) , \gamma_{\tau_{i}}(0)) + K$$
		for every $t \geq -t_{i}$, where $K, C$ are the quasi-convexity constants. Therefore, up to convergent subsequences, we get a limit of the sequence $\tau_{i}\rightarrow \tau_{\infty} \in \tilde{\mathcal{F}}^{s}(\theta)$ by the continuity of the foliation $\tilde{\mathcal{F}}^{s}$ and the fact that $\phi_{t_{i}}(\theta)$ converges to $\theta$. The point $\tau_{\infty}$ does not belong to $U_{\theta}$, and the geodesic 
		$\gamma_{\tau_{\infty}}$ satisfies 
		$$ d(\gamma_{\theta}(t) , \gamma_{\tau_{\infty}}(t)) \leq K d(\gamma_{\theta}(0) , \gamma_{\tau_{\infty}}(0)) + K$$
		for every $t \in \mathbb{R}$. This is clearly a contradiction since $I(\theta)$ is disjoint from the complement of the neighborhood $U_{\theta}$. This finishes the proof of the Proposition. 
		
	\end{proof}
	
	\section{Density of periodic orbits in a local product neighborhood }
	
	The main result of the section is the following:
	
	\begin{theorem} \label{density-per}
		Let $(M,g)$ be a compact Riemannian manifold without conjugate points such that $(\tilde{M}, \tilde{g})$ is a quasi-convex space where geodesic rays diverge. Let $\theta \in T_{1}M$ be a hyperbolic point, then periodic orbits are dense in any local product open neighborhood $W \subset T_{1}M$ in the following sense: given $I(\eta) \subset W$ and $\epsilon >0$ there exists a periodic point $\eta_{\epsilon} \in W$ of the geodesic flow and $\iota \in I(\eta)$ such that $d(\eta_{\epsilon}, \iota) <\epsilon$. 
	\end{theorem}
	
	To motivate the ideas of the proof let us briefly recall the proof of the existence of periodic orbits in hyperbolic dynamics by applying the topological index theory. Let $f$ be an Anosov diffeomorphism defined in a compact manifold. Every point has a local product structure, namely, an open neighborhood that is foliated by pieces of stable and unstable sets that is homeomorphic to a rectangle. The action of the diffeomorphism in the product neigborhood stretches the unstable leaves and contracts the stable leaves, mapping the neighborhood into a band whose width tends to zero with the iterates of the map. If the point is recurrent, the orbits of the unstable fibers by recurrent times are mapped into a thin band intersecting the product neighborhood in a subretangle whose stable width is very small. This induces a map in the space of unstable leaves in the product neighborhood that is a contraction, and by Brouwer's fixed point Theorem there is a fixed unstable leaf close to the recurrent point. The same argument proceeds with backwards iterates of the diffeomorphism, yielding the existence of a fixed stable leaf close to the recurrent point. The intersection of both fixed leaves is a periodic point. 
	
	Under the assumptions of Theorem \ref{density-per} there is no hyperbolicity a priori, so there is no local product structure in the hyperbolic sense. What we have is the local product structure claimed in Theorem \ref{local-product}, that is weaker than hyperbolic local product structure since there might be nontransversal intersections between invariant leaves. Moreover, there exist a family of foliated neighborhoods of the bi-asymptotic classes of recurrent points where the dynamics acts by contractions. These neighborhoods are not quite product neighborhoods in the dynamical sense, although they are foliated by invariant leaves.  
	
	Theorem \ref{density-per} will follow from the application of Lefchetz fixed point theory applied to the action of the dynamics on these foliated neighborhoods, as we shall see in the forthcoming subsections. 
	
	Through the section we shall assume that $(M,g)$ satisfies the hypothesis of Theorem \ref{density-per}. 
	
	\subsection{Basis of foliated neighborhoods for points in a local product neighborhood}
	
	\bigskip
	
	We consider the collection of foliated sets defined in Section 4, 
	$$ S_{U_{\theta}^{s}, \delta }^{s} = \bigcup_{\eta \in V_{\delta '}(\theta) } \tilde{\mathcal{F}}^{s}_{U_{\theta}^{s}}(\eta) $$ 
	that is a cross section for the geodesic flow foliated by open relative neighborhoods of the stable sets, and 
	
	$$\Gamma_{U_{\theta}^{s}, \delta , \sigma}^{s} = \bigcup_{\mid t \mid < \sigma} \phi_{t}( S_{U_{\theta}, \delta }^{s} ) $$
	that is an open set foliated by 
	open relative neighborhoods of the stable sets. Let us define similar objects foliated by relative open subsets of unstable sets. Let $U_{\theta}^{u} \subset \tilde{\mathcal{F}}^{u}(\theta)$ be an open coordinate neighborhood of $\theta$ in $\tilde{\mathcal{F}}^{u}(\theta)$, let $ \tilde{\mathcal{F}}^{u}_{U_{\theta}^{u}}(\eta)$ be the relative open subset of $\tilde{\mathcal{F}}^{u}(\eta)$ homeomorphic to $U_{\theta}^{u}$ according to the homeomorphism defined in Lemma \ref{Fol-neigh}, let 
	
	$$ S_{U_{\theta}^{u}, \delta }^{u} = \bigcup_{\eta \in V_{\delta }(\theta) } \tilde{\mathcal{F}}^{u}_{U_{\theta}^{u}}(\eta) $$ 
	be a u-foliated cross section for the geodesic flow and let 
	$$\Gamma_{U_{\theta}^{u}, \delta , \sigma}^{u} = \bigcup_{\mid t \mid < \sigma} \phi_{t}( S_{U_{\theta}^{u}, \delta }^{u} ) $$
	be an open set foliated by 
	open relative neighborhoods of the unstable sets. 
	
	\begin{lemma} \label{foliated-int}
		Suppose that $I(\theta)$ is contained in an open coordinate neighborhood of $T_{1}\tilde{M}$ with compact closure. Then the intersection 
		$$ \bigcap_{\sigma, \delta, U_{\theta}^{s}, U_{\theta}^{u}}(\Gamma_{U_{\theta}^{s}, \delta , \sigma}^{s}\cap \Gamma_{U_{\theta}^{u
			}, \delta , \sigma}^{s})$$ 
		is $I(\theta)$. 
	\end{lemma}

	The following remark is a consequence of Theorem \ref{local-product}.
	
	\begin{lemma} \label{crossed-int}
		Under the assumptions of Theorem \ref{density-per}, suppose that $I(\theta)$ is compact and has a local product open neighborhood $W(\theta)$ with compact closure. Then there exists $l >0$ such that 
		\begin{enumerate}
			\item For each $\eta \in W(\theta)$ there exists $\delta(\eta)>0$  such that for each $\xi \in V_{\delta(\eta)}(\eta)$ the set 
			$$\tilde{\mathcal{F}}^{s}_{U_{\eta}^{s}}(\xi) \cap \tilde{\mathcal{F}}^{cu}(\theta)$$
			is nonempty and is the bi-asymptotic class of some point provided tbat the open relative set $\tilde{\mathcal{F}}^{s}_{U_{\eta}^{s}}(\xi)$ contains an open coordinate ball of radius $\geq l$ around $\xi$. 
			\item The sets 
			$$ \Gamma_{U_{\eta}^{s}, \delta , \sigma}^{s}\cap \tilde{\mathcal{F}}^{cu}(\eta)$$ 
			are compact sets in the unstable leaf $\tilde{\mathcal{F}}^{cu}(\eta)$ containing an open relative neighborhood of $I(\eta)$ in $\tilde{\mathcal{F}}^{cu}(\eta)$.  
		\end{enumerate} 
		The same statements hold interchanging the indices $s$ and $u$ in items (1), (2). 
	\end{lemma}
	
	\begin{proof}
		The proof of item (1) is straightforward from the assumptions. Since $W(\theta)$ is an open set, given $\eta \in W(\theta)$ there exists $\delta(\eta)>0$ such that the closure of $V_{\delta(\eta)}(\eta)$ is contained in $W(\theta)$. Hence, each stable set $\tilde{\mathcal{F}}^{s}(\xi)$ for $\xi \in V_{\delta(\eta)}(\eta)$ must intersect $\tilde{\mathcal{F}}^{cu}(\theta)$ since $W(\theta)$ is a product neighborhood. Therefore, if we consider a sufficiently large (but still small) coordinate ball $U_{\xi}^{s}$ around $\xi$ in $\tilde{\mathcal{F}}^{s}(\xi)$ we must have that 
		$$U_{\xi}^{s} \cap \tilde{\mathcal{F}}^{cu}(\theta) \neq \emptyset $$ 
		consists of a compact, small set $I(\eta(\xi), \theta)$ that is a bi-asymptotic class, and by the semi-continuity of bi-asymptotic classes for any given open neighborhood $E(I(\xi(\theta), \theta))$ there exists an open neighborhood $A(\xi)$ such that for every point $\tau \in A(\xi)$ the intersection 
		$$ \tilde{\mathcal{F}}^{s}(\tau)\cap \tilde{\mathcal{F}}^{cu}(\theta)$$
		is contained in $E(I(\xi(\theta), \theta))$. By compactness, we can cover the closure of $V_{\delta(\eta)}(\eta)$ by a finite number of the neighborhoods $A(\xi)$. This implies that the size of the relative open neighborhoods $U_{\xi}^{s}$ is bounded above uniformly by some small constant as claimed.
		
		The compactness of the set 
		$$ \Gamma_{U_{\eta}^{s}, \delta , \sigma}^{s}\cap \tilde{\mathcal{F}}^{cu}(\eta)$$ is already a consequence of the argument used to prove item (1), because $W(\theta)$ is a product neighborhood for every $\eta \in W(\theta)$. The fact that it contains an open relative neighborhood of $I(\eta)$ follows from the following observation: the set 
		$$ \cup_{\xi \in \partial V_{\delta(\eta)}(\eta)} \tilde{\mathcal{F}}^{s}(\xi) \subset \partial \Gamma_{U_{\eta}^{s}, \delta , \sigma}^{s}$$ 
		is disjoint from $I(\eta)$, since $I(\eta) = \tilde{\mathcal{F}}^{s}(\eta)\cap \tilde{\mathcal{F}}^{u}(\eta)$ is in the interior of $\Gamma_{U_{\eta}^{s}, \delta , \sigma}^{s}$. Since $I(\eta)$ and $\partial \Gamma_{U_{\eta}^{s}, \delta , \sigma}^{s} \cap \tilde{\mathcal{F}}^{cu}(\eta) $ are compact sets, there exists an open relative neighborhood of $I(\eta)$ in $\tilde{\mathcal{F}}^{cu}(\theta)$ contained in $ \Gamma_{U_{\eta}^{s}, \delta , \sigma}^{s}\cap \tilde{\mathcal{F}}^{cu}(\eta)$ as we claimed.

	\end{proof}
	
	\begin{corollary}\label{basis}
		Under the assumptions of Lemma \ref{crossed-int} the collection of neighborhoods 
		$$  \Gamma_{U_{\eta}^{s}, \delta , \sigma}^{s}\cap  \Gamma_{U_{\eta}^{u}, \hat{\delta} , \hat{\sigma}}^{s}$$
		is a basis of open neighborhoods of $I(\eta)$ for every $\eta \in W(\theta)$. 
	\end{corollary}
	
	\begin{proof}
		The proof is straightforward from Lemma \ref{crossed-int}. Indeed, notice that when $\sigma, \hat{\sigma}, \delta, \hat{\delta}$ tend to $0$, the intersections $\Gamma_{U_{\eta}^{s}, \delta , \sigma}^{s}\cap  \Gamma_{U_{\eta}^{u}, \hat{\delta} , \hat{\sigma}}^{s}$ converge to $U_{\eta}^{s} \cap U_{\eta}^{u}$ that contains $I(\eta)$. Now, if we take the intersection over all open relative neighborhoods $U_{\eta}^{s}$ of $\eta $ in $\tilde{\mathcal{F}}^{s}(\eta)$ and open relative neighborhoods $U_{\eta}^{u}$ in $\tilde{\mathcal{F}}^{u}(\eta)$, what we get is exactly $I(\eta)$. By Lemma \ref{crossed-int}, the closures of the intersections $\Gamma_{U_{\eta}^{s}, \delta , \sigma}^{s}\cap  \Gamma_{U_{\eta}^{u}, \hat{\delta} , \hat{\sigma}}^{s}$ are compact and contain open neighborhoods of $I(\eta)$. 
		By the semi-continuity of $I(\eta)$ with respect to $\eta$, every open neighborhood of $I(\eta)$ contains some $\Gamma_{U_{\eta}^{s}, \delta , \sigma}^{s}\cap  \Gamma_{U_{\eta}^{u}, \hat{\delta} , \hat{\sigma}}^{s}$ for small parameters $\sigma, \delta, \hat{\sigma}, \hat{\delta}$. This shows that the collection of the above sets form a basis of open sets for $I(\theta)$. 
	\end{proof}
	
	\subsection{Basis of foliated neighborhoods in cross sections}
	
	\bigskip
	
	In this subsection we restrict the results of the previous subsection to the cross section $S_{U_{\theta}^{s}, \delta }^{s}$ that is foliated by stable sets. Let 
	$$ \mathcal{W}^{u}(\eta) = S_{U_{\theta}^{s}, \delta }^{s}\cap \tilde{\mathcal{F}}^{cu}(\eta)$$
	for $\eta \in S_{U_{\theta}^{s}, \delta }^{s}$. Notice that $\mathcal{W}^{u}(\eta)$ is preserved by the Poincar\'{e} map of the section, and that $\mathcal{W}^{u}(\eta)$ might not be a subset of  $\tilde{\mathcal{F}}^{u}(\eta)$. The sets $\mathcal{W}^{u}(\eta)$, $\eta \in S_{U_{\theta}^{s}, \delta }^{s}$, define a continuous foliation of an open, relative subset of $S_{U_{\theta}^{s}, \delta }^{s}$ by continuous, $n$-dimensional leaves such that 
	$$ I(\eta) = \tilde{\mathcal{F}}^{s}(\eta) \cap \mathcal{W}^{u}(\eta)$$ 
	for every $\eta \in S_{U_{\theta}^{s}, \delta }^{s}$. 

Let 
	$$\Phi : \Gamma_{U_{\theta}^{s}, \delta , \sigma}^{s} \longrightarrow S_{U_{\theta}^{s}, \delta }^{s}$$ 
	be the projection in $S_{U_{\theta}^{s}, \delta }^{s}$ along the orbits of the geodesic flow. This map is close to the identity since both sets are close to $\theta$.  

Given $\eta \in S_{U_{\theta}^{s}, \delta }^{s}$, let 
$$ \mathcal{W}^{u}_{U^{u}_{\eta},\delta}(\eta) = \Phi ( \tilde{\mathcal{F}}^{u}_{U^{u}(\eta)} ) \subset \mathcal{W}^{u}(\eta)$$
\begin{lemma} \label{section-product-st}
		Under the assumptions of Lemma \ref{crossed-int}, there exists a cross section $S_{U_{\theta}^{s}, \delta }^{s}$ such that each set $$\tilde{\mathcal{F}}^{s}(\eta) \cap \mathcal{W}^{u}(\xi) \neq \emptyset$$
		 for every $\eta \in S_{U_{\theta}^{u}, \delta }^{s}$. 	Moreover, for each $\eta \in S_{U_{\theta}^{u}, \delta }^{s}$, the sets 
		$$S_{U_{\theta}^{s}, \delta }^{s}\cap  \Gamma_{U_{\eta}^{s}, \delta , \sigma}^{s}\cap \Gamma_{U_{\eta}^{s}, \delta , \sigma}^{u}$$ 
		form a basis of relative open sets of $I(\eta)$ in the section $S_{U_{\theta}^{s}, \delta }^{s}$. 
		
		The following assertions hold as well:
		\begin{enumerate}
			\item given $\eta \in S_{U_{\theta}^{s}, \delta }^{s}$, there exists $\delta(\eta)>0$ and an open relative neighborhood of $I(\eta)$,  $U^{u}_{\eta} \subset \tilde{\mathcal{F}}^{u}(\eta)$ such that the set 
			$$\bar{S}_{U_{\eta}^{u},\delta(\eta)}^{u} (\eta)= S_{U_{\theta}^{s}, \delta}^{s}\cap  \Gamma_{U_{\eta}^{u}, \delta(\eta) , \sigma}^{u}$$ 
			is homeomorphic to $\mathcal{W}^{u}_{U_{\theta}^{u}, \delta }(\eta) \times V_{\delta(\eta)}(\eta)$; 
			\item  The intersection  
			$$ \bar{S}_{U_{\theta}^{u},\delta(\eta)}^{u} (\eta) \cap \tilde{\mathcal{F}}^{s}(\eta)$$ contains an open relative neighborhood $U^{s}_{\delta(\eta)}(I(\eta))$ of $I(\eta)$ in $ \tilde{\mathcal{F}}^{s}(\eta)$. 
		\end{enumerate}

	\end{lemma}
	
	\begin{proof}
		The proof just follows from Lemma \ref{crossed-int} by intersecting the neighborhood $ \Gamma_{U_{\eta}^{u}, \delta , \sigma}^{u}$ with the section $S_{U_{\theta}^{s}, \delta }^{s}$. Notice that this intersection coincides with the projection of $ \Gamma_{U_{\eta}^{u}, \delta , \sigma}^{u}$ into the section 
$S_{U_{\theta}^{s}, \delta }^{s}$ by the map $\Phi$. 
	\end{proof}

	\subsection{The action of the Poincar\'{e} map on the basis of foliated neighborhoods in a cross section}

	\bigskip
	
	Now let us look at the projection of all the sets defined in the previous subsections by the map $P : T_{1}\tilde{M} \longrightarrow T_{1}M$. We shall use the same notations for the projections of the s-foliated sets and the u-foliated sets defined before, just changing $\theta $ by $P(\theta)$. 
	
	Before stating the main result of the subsection let us introduce some notations. Let 
	$$\Phi : \Gamma_{U_{P(\eta)}^{s}, \delta , \sigma}^{s} \longrightarrow S_{U_{P(\theta)}^{s}, \delta }^{s}$$ 
	be the projection in $S_{U_{P(\theta)}^{s}, \delta }^{s}$ along the orbits of the geodesic flow. This map is close to the identity since both sets are close to $P(\theta)$. We keep the same notation for the map $\Phi$ in $T_{1}\tilde{M}$ defined in the previous subsection. Let 
	$$ \mathcal{P} : S_{U_{P(\theta)}^{s}, \delta }^{s} \longrightarrow S_{U_{P(\theta)}^{s}, \delta }^{s}$$ 
	be the Poincar\'{e} map of the section $S_{U_{P(\theta)}^{s}, \delta }^{s}$. 
	
	For $P(\eta) \in S_{U_{P(\theta)}^{s}, \delta }^{s}$, $\delta(\eta)>0$ as in item (1) of Lemma \ref{section-product-st}, let 
	$$ \hat{V}_{\delta(\eta)}(\eta) = \Phi (V_{\delta(\eta)}(\eta) ) $$ 
	that is diffeomorphic to $V_{\delta(\eta)}(\eta)$ for $\delta$ small enough. The set $\bar{S}_{U_{\theta}^{u},\delta(\eta)}^{u} (\eta)$ can be thought as a fibration by pieces of unstable sets with base at $\hat{V}_{\delta(\eta)}(\eta)$. This is a more precise description of the homeomorphism in item (1) of Lemma \ref{section-product-st}. 
	
	Let $\hat{V}_{\delta(\eta)}(P(\eta)) = P(\hat{V}_{\delta(\eta)}(\eta))$, let  $\bar{S}_{U_{P(\theta)}^{u},\delta(\eta)}^{u} (P(\eta)) =P(\bar{S}_{U_{\theta}^{u},\delta(\eta)}^{u} (\eta))$, and let   $\bar{S}_{U_{P(\theta)}^{s},\delta(\eta)}^{s} (P(\eta)) =P(\bar{S}_{U_{\theta}^{s},\delta(\eta)}^{s} (\eta))$.

	Let $$Hol_{u} : \bar{S}_{U_{P(\theta)}^{u},\delta(\eta)}^{u} (P(\eta)) \longrightarrow \hat{V}_{\delta(\eta)}(P(\eta)) $$ 
	be the projection of $\xi \in \bar{S}_{U_{P(\theta)}^{u},\delta(\eta)}^{u} (P(\eta))$ along the leaf $\mathcal{W}^{u}(P(\xi))$.  Let 
	$$Hol_{s} : \bar{S}_{U_{P(\eta)}^{s},\delta(\eta)}^{s}  \longrightarrow \hat{V}_{\delta(\eta)}(P(\eta)) $$ 
	be the projection along the stable leaves of $\mathcal{F}^{s}$ in the section $$\bar{S}_{U_{P(\eta)}^{s},\delta(\eta)}^{s}= \Phi(S_{U_{P(\eta)}^{s},\delta(\eta)}^{s})\subset S_{U_{P(\theta)}^{s},\delta}^{s}.$$
	\begin{lemma} \label{s-contraction}
		Suppose that $P(\eta) \in S_{U_{P(\theta)}^{s}, \delta }^{s}$ is a recurrent point for the Poincar\'{e} map $\mathcal{P}$. 
		Then there exists $n>0$ such that the map 
		$$Hol_{u} \circ \mathcal{P}^{n}: \hat{V}_{\delta(\eta)}(P(\eta)) \longrightarrow \hat{V}_{\delta(\eta)}(P(\eta))$$ 
		is a (topological) contraction. Similarly, there exists $m<0$ such that the map  
		$$Hol_{s} \circ \mathcal{P}^{m}: \hat{V}_{\delta(\eta)}(P(\eta) )\longrightarrow \hat{V}_{\delta(\eta)}(P(\eta))$$ 
		is a topological contraction. 
	\end{lemma}

\begin{figure}[ht]
\begin{center}
\includegraphics[scale=0.84]{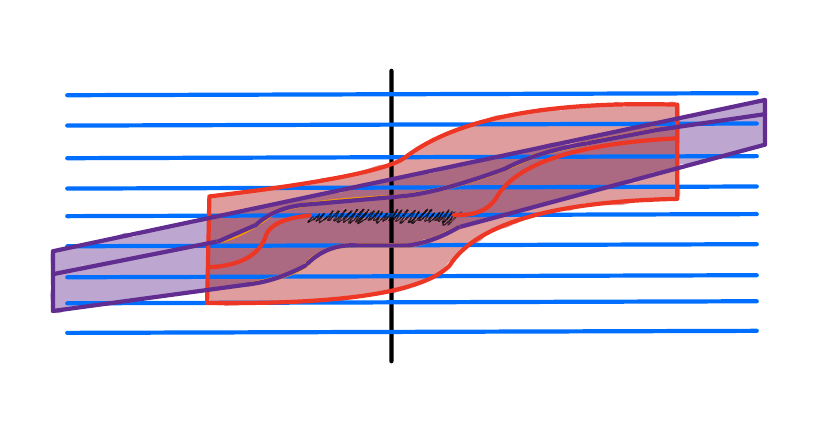}
\begin{picture}(0,0)
\put(-209,77){$I(\theta)$} 
\put(-259,102){$S^s_{U^s_{\theta,\delta}}$} 
\put(-129,102){$\mathcal{W}^u(\theta)$} 
\put(-203,149){$V_\delta(\theta)$} 
\put(-369,52){$\mathcal{P}^n(S^s_{U^s_{\theta,\delta}})$} 
\end{picture}
\end{center}
\vspace{-0.5cm}
\caption{{\small Depiction of the return map appearing in Lemma \ref{s-contraction} .}}\label{fig-3}
\end{figure}
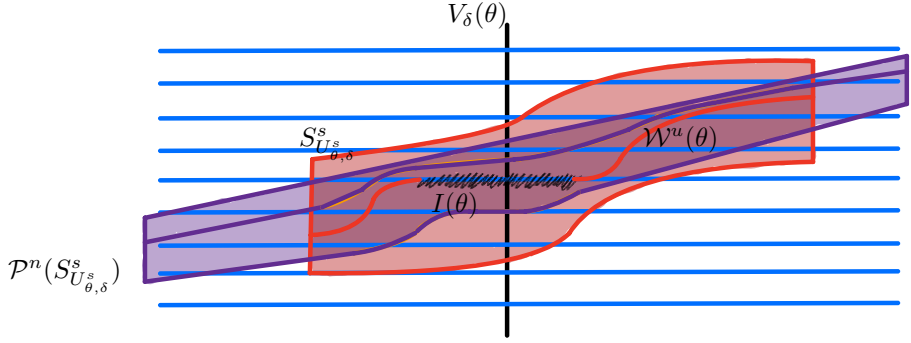

	\begin{proof}
		
		By Lemma \ref{section-product-st} item (2), the intersection $ \bar{S}_{U_{\theta}^{s},\delta(\eta)}^{u} (\eta) \cap \tilde{\mathcal{F}}^{s}(\eta)$ is a compact set of the section $S_{U_{\theta}^{s}, \delta }^{s}$ that contains an open relative neighborhood $U^{s}_{\delta(\eta)}(I(\eta))$ of $I(\eta)$ in $ \tilde{\mathcal{F}}^{s}(\eta)$.
		
		By Lemma \ref{contraction}, there exists $n >0$ such that the $n-th$-iterate $\mathcal{P}^{n}$ of the Poincar\'{e} map of the section $S_{U_{P(\theta)}^{s}, \delta }^{s}$ applied to $\bar{S}_{U_{\theta}^{s},\delta(\eta)}^{u} (P(\eta)) \cap \mathcal{F}^{s}_{c}(P(\eta))$, where $\mathcal{F}^{s}_{c}(P(\eta))$ is the connected component of $\mathcal{F}^{s}(P(\eta))$ in the section $S_{U_{P(\theta)}^{s}, \delta }^{s}$ containing $P(\eta)$, is a subset of $U^{s}_{\delta(\eta)}(I(P(\eta)))$. Since the map $\mathcal{P}$ preserves the sets $\mathcal{W}^{u}(\xi)$, this implies that the image by $\mathcal{P}^{n}$ of the leaves $\mathcal{W}^{u}(\xi)$ in $\bar{S}_{U_{\theta}^{s},\delta(\eta)}^{u} (P(\eta))$ only meet $\hat{V}_{\delta(\eta)}(P(\eta))$ in its interior.
		
		In other words, 
		$$Hol_{u}\circ \mathcal{P}^{n}(\bar{S}_{U_{\theta}^{s},\delta(\eta)}^{u} (P(\eta))) \subset \hat{V}_{\delta(\eta)}(P(\eta)) $$ 
		and in particular, 
		$$Hol_{u}\circ \mathcal{P}^{n}(\hat{V}_{\delta(\eta)}(P(\eta))) \subset \hat{V}_{\delta(\eta)}(P(\eta)). $$
		This shows that $Hol_{u}\circ \mathcal{P}^{n}$ is a topological contraction in $\hat{V}_{\delta(\eta)}(P(\eta))$ as claimed. 	
		
		The proof of the second part of the statement follows from the above argument just observing that the involution $(p,v) \rightarrow (p,-v)$ interchanges stable and unstable sets. 
		\end{proof}

	\subsection{Lefchetz index theory and the existence of periodic points in a local product neighborhood}
	
	\bigskip
	
	The purpose of the subsection is to apply Lefchetz fixed point theory to the action of the Poincar\'{e} map $\mathcal{P}$ in the foliated set 
$$\bar{S}_{U_{P(\eta)}^{u},\delta(\eta)}^{u} (P(\eta))= S_{U_{P(\theta)}^{s}, \delta}^{s}\cap  \Gamma_{U_{P(\eta)}^{u}, \delta(\eta) , \sigma}^{u}$$ 
defined in item (1) of Lemma \ref{section-product-st} that is homeomorphic to the product of open disks $\mathcal{W}^{u}_{U_{P(\eta)}^{u}, \delta }(P(\eta)) \times V_{\delta(\eta)}(P(\eta))$.  
	
\begin{lema}[Lefschetz fixed point theorem] \label{fixed-points}
	Let $U = \mathbb{D}^{\ell_1} \times \mathbb{D}^{\ell_2}$ with $d= \ell_1 + \ell_2$ and $f: U \to \RR^d$ be a continuous map so that $f(U) \cap U \neq \emptyset$ and:
	\begin{itemize}
		\item $f(\partial \mathbb{D}^{\ell_1} \times \mathbb{D}^{\ell_2}) \cap U = \emptyset$ ,
		\item $\mathbb{D}^{\ell_1} \times \partial \mathbb{D}^{\ell_2} \cap f(U) = \emptyset$. 
	\end{itemize} 
	Then, there is a fixed point of $f$ inside $U$. 
\end{lema}

\begin{proof}
	By hypothesis, we have that after composing $f$ with an isotopy of $\RR^d$ which is the identity on $U$ (thus does not affect the fixed point set of $f$) we can assume that for every $z=(x,y)$ with $\|x\|=1$ (i.e. $z \in \partial \mathbb{D}^{\ell_1} \times \mathbb{D}^{\ell_2}$ we have that if $z'=(x',y')=f(z)$ then $\|x'\|>r>1$. Similarly, if $z= (x,y)$ with $\|y\|=1$ then we have that $\|y'\|<1/r<1$. This holds for some $r>1$ by compactness.

	Define $R$ to be a $\delta$-neighborhood of $\partial U$ with the property that for $z=(x,y)$ with $\|x\|\in (1-\delta,\delta)$ one has that $z'=(x',y')=f(z)$ verifies that $\|x'\|> r$. Similarly, for $z=(x,y)$ with $\|y\| \in (1-\delta,1)$ we have that $\|y'\| < 1/r$ that by choosing $\delta$ small we can assume that $1/r<(1-\delta)$. 
	
	Consider $\rho: U \to [0,1]$ a smooth function which is constant equal to $1$ in $U \setminus R$ and constant equal to $0$ in a small neighborhood of $\partial U$ (clearly, contained in $R$). Define $f_t: U \to \RR^d$ defined by $f_t(z) = (1-t) f(z) + t f_\rho(z)$ where
	
	$$ f_\rho(z)= (1-\rho(z)) f(z) + \rho(z) Az$$ \noindent where $Az= A(x,y) = (\lambda x, \lambda^{-1} y)$ and $\lambda>1$ is close to $1$. It follows that the segment joining $f(z)$ and $Az$ does not contain $z$ for every $z \in R$ because either the first coordinate is strictly larger for both, or the second coordinate is strictly smaller for both. 
	
	We claim that the only fixed point of $f_\rho$ is $(0,0)$ and has Lefschetz index with absolute value $1$. 
	
	To see this, note that $f_\rho = A$ in $U \setminus R$ so it cannot have fixed points other than $(0,0)$ there. Clearly, the index of $(0,0)$ is either $1$ or $-1$ depending on the dimensions $\ell_1$ and $\ell_2$.  We must then see that there are no fixed points in $R$, but this follows directly because $z$ is not contained in the segment between $f(z)$ and $Az$ as remarked before. 
	
	Therefore, since $f$ is homotopic to $f_\rho$ with a homotopy which is constant in a neighborhood of $\partial U$ we deduce that the total Lefchetz index of $f$ in $U$ is non-zero and therefore has a fixed point. 
\end{proof}
	
\begin{corollary} \label{fixed-points-dense}
Let $P(\eta)$ be a recurrent point for the Poincar\'{e} map $\mathcal{P} : S_{U_{P(\theta)}^{s}, \delta}^{s}\longrightarrow S_{U_{P(\theta)}^{s}, \delta}^{s} $,  let $\delta(\eta)>0$,  $U_{\eta}^{u},$ be given in Lemma \ref{section-product-st}. Then for every $\tau < \delta(\eta)$, and every relative open neighborhood of $I(P(\eta))$, $A_{P(\eta)}^{u} \subset U_{P(\eta)}^{u}=P(U_{\eta}^{u})$
homeomorphic to a disk, the set 
	$$\bar{S}_{A_{P(\eta)}^{u},\tau}^{u} (P(\eta))= S_{U_{P(\theta)}^{s}, \delta}^{s}\cap  \Gamma_{A_{P(\eta)}^{u}, \tau , \sigma}^{u}$$ 
contains a periodic point of the Poincar\'{e} map. In particular, given $\epsilon >0$ there exists a periodic point of the Poincar\'{e} map that is within a distance $\epsilon$ from some point in $I(P(\eta))$. 
\end{corollary}

\begin{proof}
By Lemma \ref{section-product-st}, the set $\bar{S}_{A_{P(\eta)}^{u},\tau}^{u} (P(\eta))= S_{U_{P(\theta)}^{s}, \delta}^{s}\cap  \Gamma_{A_{P(\eta)}^{u}, \tau , \sigma}^{u}$ is homeomorphic to the product 
$$ \mathcal{W}^{u}_{A_{P(\eta)}^{u}, \delta} \times V_{\tau}(P(\eta)) $$ 
or equivalently, homeomorphic to the product 
$$ \mathcal{W}^{u}_{A_{P(\eta)}^{u}, \delta }\times \Phi(V_{\tau}(P(\eta)) ) \subset S_{U_{P(\theta)}^{s}, \delta}^{s}$$ 
that is the product of two  $n$-dimensional disks of the section $S_{U_{P(\theta)}^{s}, \delta}^{s}$. Now we apply Lemma \ref{fixed-points} to the Poincar\'{e} map map $f = \mathcal{P}^{n}$  and the set 
$$U = \mathbb{D}_{1}^{k-1} \times \mathbb{D}_{2}^{k-1}$$
with $\mathbb{D}_{1}^{k-1} = \mathcal{W}^{u}_{A_{P(\eta)}^{u}, \delta }$, $\mathbb{D}_{2}^{k-1} = \Phi(V_{\tau}(P(\eta)) )$, $k = dim(M)$, up to a change of coordinates. 

Indeed, by Lemma \ref{s-contraction} and Proposition \ref{contraction} applied to  relative open neighborhoods of $I(\psi)$ in $\mathcal{W}^{u}(\psi)$, the map $\mathcal{P}^{n}$ applied to the set $U$ satisfies the assumptions of Lemma \ref{fixed-points} and therefore, it contains a fixed point for $\mathcal{P}^{n}$. 

Now, consider a sequence $0<\tau_{j} <\delta(\eta)$ converging to zero, and a sequence of sets $\mathcal{W}^{u}_{A_{P(\eta),j}^{u}, \delta }$ whose intersection is the set $I(P(\eta))$ (this is possible by Corollary \ref{basis}). We can apply Lemma \ref{fixed-points} to some power $\mathcal{P}^{n_{j}}$ of the Poincar\'{e} map to the set 
$$ U_{j} = \mathcal{W}^{u}_{A_{P(\eta),j}^{u}, \delta }\times \Phi(V_{\tau_{j}}(P(\eta)) ) \subset S_{U_{P(\theta)}^{s}, \delta}^{s}$$
to conclude that there is a fixed point of $\mathcal{P}^{n_{j}}$ in $U_{j}$. This finishes the proof of the Corollary.

\end{proof}

	\section{The proof of the main Theorem and further results: topological horseshoes and entropy} 

In this final section we present the proof of Theorem \ref{main} and further applications of the Lefchetz fixed point Theory to geodesic flows without conjugate points and continuous horospherical foliations, namely, Theorem \ref{main2}. 

Let us start with the proof of Theorem \ref{main}. The theorem claims that if the geodesic flow of a compact, connected manifold $(M,g)$ without conjugate points with quasi-convex universal covering where geodesic rays diverge is $C^{2}$-structurally stable from Ma\~{n}\'{e}'s viewpoint, then the geodesic flow is Anosov. We know by Theorem \ref{persistence} that the $C^{2}$-structural stability from Ma\~{n}\'{e}'s viewpoint implies implies that the closure of the set of periodic orbits is a hyperbolic set. 
	
		By Corollary \ref{fixed-points-dense}, there exists an open neighborhood of the closure of periodic orbits where the bi-asymptotic class of every recurrent point is approached by periodic orbits. Since recurrent points have total Liouville measure in $T_{1}M$ this yields that the whole neighborhood is contained in the closure of the set of hyperbolic periodic points, a hyperbolic set. Therefore, the set of points in $T_{1}M$ contained in a hyperbolic set (with uniform constants) of the geodesic flow is open, and since it is also closed it must be the whole manifold if the manifold is connected, so the geodesic flow is Anosov. 

Of course, if the geodesic flow is Anosov the geodesic flow is $C^{1}$ structurally stable by the work of Anosov and in particular, it is $C^{2}$-structurally stable from Ma\~{n}\'{e}'s viewpoint. This completes the proof of Theorem \ref{main}.
	
	\bigskip

Let us now prove Theorem \ref{main2}, 	whose assumptions do not include the stability of the geodesic flow of $(M,g)$. A crucial step of the proof is the following stament:

\begin{prop} \label{homoclinic}
Let $(M^{n},g)$ be a compact $C^{\infty}$, $n$-dimensional, connected Riemannian manifold without conjugate points  such that the universal covering is a quasi-convex space where geodesic rays diverge. Suppose that there exists a hyperbolic periodic point $\theta$ for the geodesic flow. Then there exists a homoclinic intersection between the stable manifold $\mathcal{F}^{s}(\theta)$ and the center unstable manifold $\mathcal{F}^{cu}(\theta)$. 
Analogously, there exists a homoclinic intersection between the unstable manifold $\mathcal{F}^{u}(\theta)$ and the center stable manifold $\mathcal{F}^{cs}(\theta)$. 
\end{prop} 
	
Recall that a {\emph homoclinic intersection} between $\mathcal{F}^{s,u}(\theta)$ and the center unstable manifold $\mathcal{F}^{cu,s}(\theta)$ is a point $\eta \in \mathcal{F}^{s,u}(\theta) \cap \mathcal{F}^{cu,s}(\theta)$ that is different from $\theta$. If these submanifolds are transverse at $\eta$, the point $\eta$ is called a {\emph transverse} homoclinic point. 
So Proposition \ref{homoclinic} claims that there exists a homoclinic intersection between the invariant submanifolds of the hyperbolic periodic point $\theta$, not necessarily transverse. 

The existence of transverse homoclinic points of hyperbolic periodic orbits is one of the main reasons of chaotic dynamics in the sense of entropy: their existence implies positive topological entropy and exponential growth of the set of periodic orbits in terms of their periods. We shall extend this idea to our context despite the possible lack of transversality, as a consequence of the results of Sections 5 to 7. 
\bigskip

{\textbf Proof of Proposition \ref{homoclinic}}
\bigskip

Let $\theta \in T_{1}M$ be a hyperbolic periodic point of the geodesic flow. By the results of Section 5, there exists an open neighborhood $W(\theta)$ with a local product structure.  We can suppose without loss of generality that $W(\theta)$ is a $u$-foliated neighborhood by the remarks of Section 4 (Lemma \ref{vertical-neigh}, namely, it is foliated by open relative subsets of the leaves of $\mathcal{F}^{cu}$ homeomorphic to $(n-1)$ dimensional disks where $n = dim(M)$. Let $\alpha \in W(\theta)$ be a (forward and backwards) recurrent point, and let $t_{n} \rightarrow +\infty$ be such that $\phi_{t_{n}}(W(\theta) )\cap W(\theta) \neq \emptyset $. By the invariance of the foliation $\mathcal{F}^{cu}$, the leaves 
$$ \phi_{t_{n}}(\mathcal{F}^{cu}(\alpha) )= \mathcal{F}^{cu}(\phi_{t_{n}}(\alpha))$$ 
of some points $\beta \in W(\theta)$ come back to $W(\theta)$, that is a local product neighborhood of $\theta$. Therefore, the connected component in $W(\theta)$ of the center unstable leaf $\mathcal{F}^{cu}(\phi_{t_{n}}(\alpha)$ which contains $\phi_{t_{n}}(\alpha)$ meets $\mathcal{F}^{s}(\theta)$ at a homoclinic point $\eta$. An analogous argument applies to backward recurrency and $s$-foliated open neighborhoods of $\theta$, thus finishing the proof of the Proposition. 
\bigskip

\begin{lemma} \label{hom-acc}
Under the assumptions of Proposition \ref{homoclinic}, the periodic point $\theta$ is accumulated by homoclinic intersections of the submanifolds $\mathcal{F}^{s,u}(\theta) \cap \mathcal{F}^{cu,s}(\theta)$. 
\end{lemma}

\begin{proof}
The proof follows from Proposition \ref{homoclinic}, the contraction of the stable manifold $\mathcal{F}^{s}(\theta)$ by $\phi_{t}$, $t>0$, and the contraction of the unstable manifold $\mathcal{F}^{u}(\theta)$ by $\phi_{t}$, $t <0$. These facts imply that the homoclinic point $\eta \in \mathcal{F}^{s}(\theta) \cap \mathcal{F}^{cu}(\theta)$ has a sequence $\phi_{s_{n}}(\eta) \in \mathcal{F}^{s}(\theta)$ of points in its orbit, endowed with a sequence of relative open $n$-dimensional disks $U^{u}_{n}(\phi_{s_{n}}(\eta)) \subset \mathcal{F}^{cu}(\phi_{s_{n}}(\eta))$ which are relative neighborhoods of $\phi_{s_{n}}(\eta)$, converging to a local disk in $\mathcal{F}^{cu}(\theta)$ in the $C^{0}$ topology.  Moreover, for a homoclinic point $\beta \in \mathcal{F}^{u}(\theta) \cap \mathcal{F}^{cs}(\theta)$, there exists a sequence $\phi_{t_{n}}(\beta) \subset \mathcal{F}^{u}(\theta)$ and a sequence $U^{s}_{n}(\phi_{t_{n}}(\beta)) \subset \mathcal{F}^{s}(\phi_{t_{n}}(\beta))$ of $n$ dimensional disks converging in the $C^{0}$ topology to a disk in $\mathcal{F}^{cs}(\theta)$ containing $\theta$. Since $W(\theta)$ is a local product neighborhood of $\theta$, every disk $U^{u}_{n}(\phi_{s_{n}}(\eta))$ meets every disk $U^{s}_{n}(\phi_{t_{n}}(\beta))$ for $\phi_{s}(\eta), \phi_{t_{n}}(\beta) \in W(\theta)$, thus proving the lemma.

\end{proof}

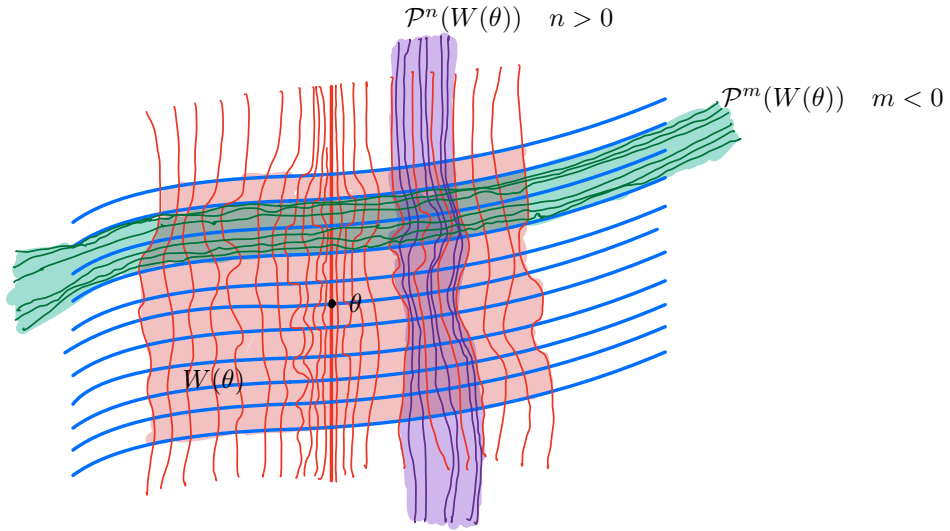
\begin{figure}[ht]
\begin{center}
\includegraphics[scale=0.64]{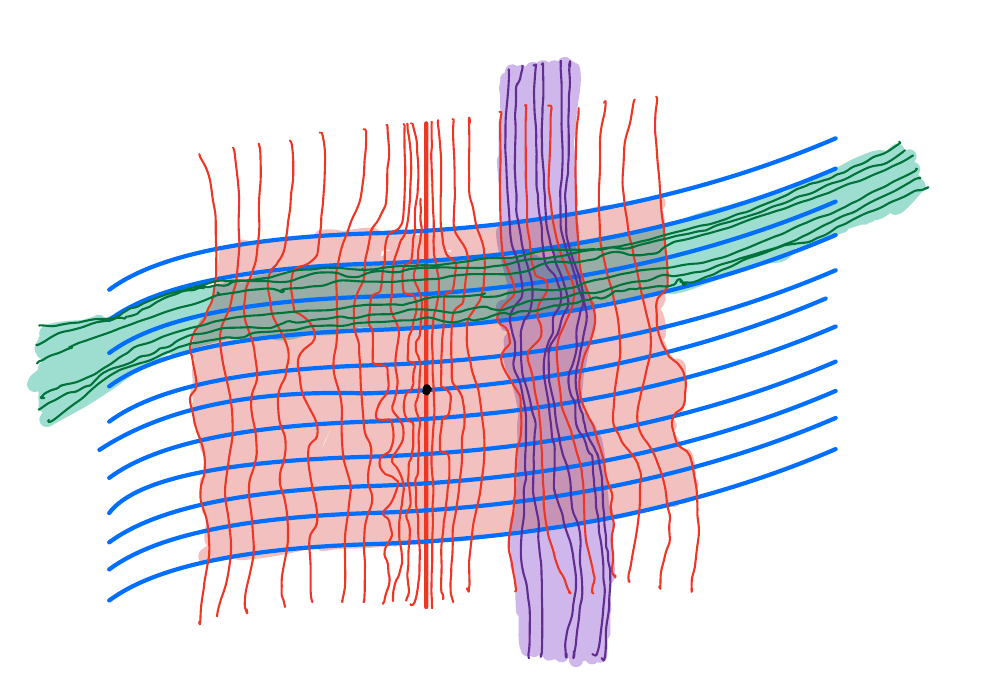}
\begin{picture}(0,0)
\put(-169,92){$\theta$} 
\put(-232,63){$W(\theta)$}
\put(-148,200){$\mathcal{P}^n(W(\theta)) \quad n >0$}
\put(-28,170){$\mathcal{P}^m(W(\theta)) \quad m<0$}
\end{picture}
\end{center}
\vspace{-0.5cm}
\caption{{\small  Construction of a homoclinic point that produces a horseshoe. }}\label{fig-4}
\end{figure}

{\textbf Proof of Theorem \ref{main2}}
\bigskip 

Let $\theta \in T_{1}M$  be a hyperbolic periodic point for the geodesic flow of $(M,g)$, satisfying the assumptions of Theorem \ref{main2}. Let $W(\theta)$ be a local product open neighborhood, let $\bar{W}(\theta)$ be the closure of the union of $W(\theta)$ with all the heteroclinic intersections of is points, namely, the intersections of the type $\mathcal{F}^{s}(\eta) \cap \mathcal{F}^{cu}(\psi) \cap W(\theta)$, and $\mathcal{F}^{u}(\eta) \cap \mathcal{F}^{cs}(\psi) \cap W(\theta)$, for $\eta , \psi \in W(\theta)$. The set $\bar{W}(\theta)$ is a  closed set with nonempty interior, that is foliated by both connected subsets of the stable foliation $\mathcal{F}^{s}$ and the unstable foliation $\mathcal{F}^{u}$. 

By the remarks in the begining of Section 6 we can restrict the dynamics to a local cross section $S^{s}_{U^{s}_{\theta}, \delta}$ of the geodesic flow containing $\theta$ that is foliated by relative open subsets of the leaves of $\mathcal{F}^{s}$ homeomorphic to disks. By Lemma \ref{hom-acc} and the hyperbolicity of $\theta$, we have that the forward and backward orbits of $\bar{W}(\theta)$ intersect near $\theta$ in a sequence $\Omega_{n}$ of nested, foliated  subsets whose dynamics can be codified by a shift of finite type, like in the case of the classical transverse Smale's horseshoe. The details of the proof of the above assertion can be found in \cite{kn:KY}, \cite{kn:Gromov2} for instance, where topological horseshoes were already studied. Their intersection $\Omega$ is a compact nonempty, invariant subset that resembles a transverse horseshoe, the difference being the possible existence of more than one point in each connected component of $\Omega$. We can apply Lemma \ref{fixed-points} to each one of the connected components of $\Omega_{n}$ to conclude that $\Omega$ is in the closure of the set of periodic orbits. 
 The topological entropy of the flow restricted to $\Omega$ is at least the entropy of the correspondent shift (the dynamics in $\Omega$ is in general semi-conjugate to the shift dynamics). This concludes the proof of Theorem \ref{main2}.



\begin{thebibliography}{10}
		
		\bibitem{kn:Anosov} Anosov, D. V.: \textit{Geodesic flows on closed Riemannian manifolds of negative curvature}.
		Proceedings of the Steklov Institute of Mathematics, No. 90 (1967). Translated from the Russian by S. Feder American Mathematical Society, Providence, R.I. 1969 iv+235 pp. 57.50 (53.00)
		
		\bibitem{kn:Arnold} Arnold, V. I.: \textit{Mathematical Methods of Classical Mechanics}. Graduate Texts in Mathematics, 60, Second Edition, Springer. ISBN 0-387-96890-3. 
		
		\bibitem{kn:BGS} Ballmann, W., Gromov, M., Schroeder, V.: \textit{Manifolds of nonpositive curvature}. Progress in Mathematics, 61 Birkhausser, Boston, MA.
		1985. vi + 263 pp. ISBN 081763181X.
		
	\bibitem{kn:BI} Burago, D., Ivanov, S. :\textit{Riemannian tori without conjugate points are flat}. Geom. and Funct. Anal. GAFA, 4 (1994) 259-269.

		\bibitem{kn:Burns} K.~Burns, \textit{The flat strip theorem fails for surfaces with no conjugate points}, Proc. Amer. Math. Soc. \textbf{115} (1992), 199--206.

		\bibitem{kn:Busemann50} Busemann, H.: \textit{The Geometry of geodesics}, Academic Press Inc., New York, 1955. x + 422 pp.
		
		\bibitem{kn:CM} Contreras, G, Mazzuchelli, M: \textit{Proof of the $C^2$-stability conjecture for geodesic flows on compact surfaces} Duke Math. J. \emph{to appear}. 
		
		\bibitem{kn:CS} Croke, C., Schoreder, V.: \textit{The fundamental group of compact manifolds without conjugate points}. Comment. Math. Helv. 61 (1986) n. 1,
		161-175.

\bibitem{kn:Dinaburg} Dinaburg, E. I.: \textit{A connection between various entropy characterizations of dynamical systems}. Izv. Akad. Nauk SSSR Ser. Mat. 35 (1971), 324–366. 
		
		\bibitem{kn:Eberlein} Eberlein, P: \textit{Geodesic flows in certain manifolds without conjugate points}, Trans. Amer. Math. Soc. \textbf{167} (1972), 151--170.
		
		\bibitem{kn:Eberlein-73} Eberlein, P: \textit{When is a geodesic flow of Anosov type? I.} J. Diff. Geom. 8 (3) (1973) 437-463.
		
		\bibitem{kn:Eberlein-96} Eberlein, P.: \textit{Geometry of nonpositively curved manifolds.} Chicago Lectures in Mathematics. University of Chicago Press, Chicago, IL, 1996. vii+449 pp.
		
		\bibitem{kn:Eschenburg} Eschenburg, J. H.: \textit{Horospheres and the stable part of the geodesic flow}, Math. Z. \textbf{153} (1977), 237--251.

\bibitem{kn:GR1} Gelfert, K., Ruggiero, R.: \textit{Geodesic flows modelled by expansive flows}. Proceedings of the Edinburgh Mathematical Society , Volume 62 , Issue 1 (2019 ), pp. 61 - 95

\bibitem{kn:GR2} Gelfert, K., Ruggiero, R.: \textit{Geodesic flows modelled by expansive flows: compact surfaces without conjugate points and continuous Green bundles}. Ann. de l'Institut Fourier, Vol. 73 (2023), no. 6, 2605-2649. 
		
\bibitem{kn:Green} Green, L. W.: \textit{Surfaces without conjugate points}, Trans. Amer. Math. Soc. \textbf{76} (1954), 529--546.


		\bibitem{kn:Gromov} Gromov, M.: \textit{ Hyperbolic Groups.} Essays in Group Theory 8, 75-264 S. M. Gersten Editor, Springer-Verlag.
		
		\bibitem{kn:Gromov2} Gromov, M.: \textit{Three remarks on geodesic dynamics and fundamental group}. Enseignement Math\'{e}matique (2000) Swets and Zeitlinger Editors. 

\bibitem{Hirsch} M. Hirsch : \textit{Differential Topology}.  Graduate Texts in Mathematics, Vol. 33. Springer New York. 
		
		\bibitem{kn:HPS} Hirsch, M., Pugh, C., Shub, M.: \textit{Invariant Manifolds}. Lecture Notes in Mathematics, Vol. 583, Springer. ISBN-13: 978-3540081487.

\bibitem{kn:Hopf} Hopf, E.: \textit{Closed Surfaces Without Conjugate Points}. Proc. Nat. Acad. of Sci. 34, (1948), 47-51.

\bibitem{kn:Irie} Irie, K.: \textit{Dense existence of periodic Reeb orbits and ECH spectral invariants}. Journal of Modern Dynamics, Vol. 9 (2015), 357-363. 

\bibitem{kn:KY} Kennedy, J., Yorke, : \textit{Topological horseshoes}. Transactions of the AMS, 353(6) (2001), 2513-2530.
		
		\bibitem{kn:Klingenberg} Klingenberg, W.: \textit{Riemannian manifolds with geodesic flows of Anosov type}, Ann. of Math. (2) \textbf{99} (1974), 1--13.
		
		\bibitem{kn:ARR} Lazrag, A., Rifford, L., Ruggiero, R.: \textit{Franks' Lemma for $C^{2}$ Ma\~{n}\'{e} perturbations of Riemannian metrics and applications to persistence}. Journal of Modern Dynamics vol. 10, 2 (2016), 319-411.


\bibitem{kn:Liao} Liao, S. T.: \textit{On the stability conjecture}. Chinese Ann. of Math. 1  (1980), 9-30.  

\bibitem{kn:MR} Mamani, E., Ruggiero, R.: \textit{Expansive factors for geodesic flows of compact manifolds without conjugate points and with visibility universal covering}, arXiv:2311.02698 
		
		\bibitem{kn:Mane} Ma\~{n}\'{e}, R.: \textit{An ergodic closing lemma}. Ann. of Math. (2) 116 (1982), n. 3, 503-540.
		
		\bibitem{kn:Mane88} Ma\~n\'e, R.: \textit{A proof of the $C^1$ stability conjecture} Inst. Hautes \'Etudes Sci. Publ. Math., 66:161--210, 1988.
		
		\bibitem{kn:Mane2} Ma\~{n}\'{e}, R.: \textit{On a theorem of Klingenberg.} Dynamical systems and biffurcation theory. M. Camacho, M. Pac\'{\i}fico, F. Takens Editors.
		Pitman Research Notes in mathematics 160 (1987) 319-345.

\bibitem{kn:Manning} Manning, A.: \textit{Topological entropy for geodesic flows}. Ann. of Math. 110 (1979) 567-573. 

\bibitem{kn:Milnor}  Milnor, J, \textit{A note on curvature and fundamental group}. J. of Diff. Geom. 2 (1968) 1-17.

\bibitem{kn:Morse} Morse, H. M.: \textit{A fundamental class of geodesics on any closed surface of genus greater than one}. Trans. Amer. Math. Soc. Vol. 26, 1 (1924), 25-60.
		
		\bibitem{kn:Newhouse} Newhouse, S.: \textit{Quasi-elliptic periodic points in conservative dynamical systems.} Amer. J. Math. 99(5):1061--1087, 1977.
		
\bibitem{kn:Paternain} Paternain, G.: \textit{Geodesic flows}. Progress in Mathematics, Birkhausser, Boston, MA. 
		\bibitem{kn:Pesin} Pesin, Y.: \textit{Geodesic flows on closed Riemannian manifolds without focal points}.  Math USSR Izvestija Vol. 11 (1977) N. 6, 1195-1228. .
		
		\bibitem{kn:Pugh} Pugh, C., Robinson, C.: \textit{$C^{1}$ closing lemma, including Hamiltonians}. Ergodic Theory Dynam. Systems, 3:261--313, 1983.
		
		\bibitem{kn:Ruggiero-AHL} Rifford, L., Ruggiero, R.: \textit{ On the stability conjecture for geodesic flows of manifolds without conjugate points}. Annales Henri Lebesgue, 4: 759--784, 2021. 
		
		\bibitem{kn:Robin-Robinson} Robinson, C.: \textit{Generic properties of conservative systems I and II.} Amer. J. Math.,  92:562--603 and 897--906, 1970.
		
		\bibitem{kn:Ruggiero-SBM} Ruggiero, R.: \textit{Expansive dynamics and hyperbolic geometry}. Bol. Soc. Bras. Mat. vol. 25 (1994) 139-172.
		
		\bibitem{kn:Ruggiero-expansive}  Ruggiero, R. : \textit{ Expansive geodesic flows in manifolds without conjugate points}.
		Erg. Theory and Dyn. Systems vol. 17 (1997) 211-225.
		
		\bibitem{kn:Ruggiero-Ast} Ruggiero, R.: \textit{On the divergence of geodesic rays in manifolds without conjugate points,
			dynamics of the geodesic flow and global geometry}. Ast\'{e}risque 287 (2003) 231-249.
		
		\bibitem{kn:Ruggiero-Ensaios} Ruggiero, R.: \textit{ Dynamics and global geometry of manifolds without conjugate points}. Ensaios Matem\'{a}ticos vol. 12 (2007), 1-181.
		Sociedade Brasileira de Matem\'{a}tica. ISSN 2175-0432.
		
		\bibitem{kn:Ruggiero-TAMS} Ruggiero, R. : \textit{Does the hyperbolicity of periodic orbits of a geodesic flow without conjugate points imply the Anosov property?} Transactions of the AMS, 374 (10) 
		
	\end{thebibliography}
\end{document}